\documentclass[preprint,12pt]{elsarticle}

\usepackage{amssymb}
\usepackage{comment}
\usepackage{comment}
\usepackage{here}
\usepackage{amsmath}
\usepackage{amsthm}
\usepackage[shortlabels,inline]{enumitem}
\usepackage{booktabs}
\usepackage{hyperref}
\usepackage{xcolor}
\usepackage{subcaption}
\usepackage{todonotes}
\usepackage{tikz}
\setlist[enumerate]{leftmargin=15pt}
\setlist[description]{leftmargin=10pt,topsep=1pt}
\setlist[itemize]{leftmargin=15pt}

\makeatletter

\let\origQed\qed
\newcommand*{\resetqed}{\gdef\qed{\origQed\global\let\qed\relax}}
\resetqed
\let\origendproof\endproof
\def\endproof{\origendproof\resetqed}

\newtheorem{theorem}{Theorem}
\newtheorem*{conjecture}{Conjecture}

\newtheorem{lemma}[theorem]{Lemma}
\newtheorem{definition}[theorem]{Definition}
\newtheorem{corollary}[theorem]{Corollary}
\newtheorem{remark}[theorem]{Remark}
\newtheorem{strategy}{Strategy}
\newenvironment{claim}{\begin{proof}[Claim.]}{\let\qed\relax\end{proof}}
\newcommand{\ZZ}{\mathbb{Z}}
\newcommand{\NN}{\mathbb{N}}

\newcommand{\case}[3][\relax]{%
  \item[Case #2:] #3.\par%
  \edef\test{#1}%
  \if\test\relax\else%
    \leavevmode%
    \marginpar{%
      \scalebox{.8}{%
        \begin{tikzpicture}[baseline=-\baselineskip]
          \path[use as bounding box] ;
          \node[draw,anchor=north west] at (0,0) {\begin{tabular}{@{}l@{}}#1\end{tabular}};
        \end{tikzpicture}%
      }%
    }%
  \fi%
}

\journal{Theoretical Computer Science}

\begin{document}

\begin{frontmatter}

\title{On the Metric Dimension of 
$ K_a \times K_b \times K_c $}

\author[cs,math]{Valentin Gledel}
\author[math]{Gerold Jäger}

\affiliation[cs]{organization={Department of Computing Science, Umeå University},
            city={Umeå},
            country={Sweden}}
\affiliation[math]{organization={Department of Mathematics and Mathematical Statistics, Umeå University},
            city={Umeå},
            country={Sweden}}

\begin{abstract}
In this work we determine the metric dimension 
of $ K_a \times K_b \times K_c$ 
for all $a,b,c\in \NN$ with $ a \le b \le c $ as follows.
For $3a<b+c$ and $2b \le c$, this value is $c-1$,
for $3a<b+c$ and $2b > c$, it is $\left \lfloor \frac{2}{3}(b+c-1) \right \rfloor$, and for $3a=b+c$, 
it is $\left \lfloor \frac{a+b+c}{2} \right \rfloor -1 $.
The only open case is  $3a>b+c$, where two values are possible, namely
$\left \lfloor \frac{a+b+c}{2} \right \rfloor -1 $
and $\left \lfloor \frac{a+b+c}{2} \right \rfloor $.
This result extends previous results of~\cite{CHMPPSW07}, who computed the metric dimension
of $ K_a \times K_b$, and of~\cite{JD20}, who computed the metric dimension of 
$ K_a \times K_a \times K_a$\footnote{In comparison to~\cite{JD20} we 
use the notation $K_a$ instead of $\ZZ_a$, as this notation of complete graphs is more common in graph theory.} 

We prove our result by introducing and analyzing a new 
variant of Static Black-Peg Mastermind, 
in which each peg has its own permitted set of colors.
For all cases, we present strategies 
which we prove to be both feasible and optimal.
Our main result follows, as
the number of questions of these strategies is
equal to the metric dimension of $K_a \times K_b \times K_c$. 
\end{abstract}


\end{frontmatter}

\section{Introduction}

The metric dimension is
an important property of an undirected and unweighted graph $G=(V,E)$. 
It is defined as the minimum size of a set $ U \subseteq V $ so that every vertex $v \in V $
is uniquely determined by the vector of distances between $v$ and the vertices in $U$. 
The concept of metric dimension has been mentioned first 
in~\cite{Blu53} and was then independently re-introduced in~\cite{HM76,Sla75}. 
The metric dimension has been studied for several graph classes, 
see e.g.~\cite{CEJO00,FHHMS15,YKR11,RKYS15}.
Determining whether the metric dimension of a graph is smaller than or equal to a given value
is $\mathsf{NP}$-complete~\cite{GJ79,KRR96}.

In this paper, we aim to determine the metric dimension of a cross product of 
complete graphs. The following results are know for this problem.
The metric dimension  of $ K_p^{c} = K_p \times K_p \times \dots \times K_p $ 
is in $\Theta(p/{\log p})$,
if $c$ is constant~\cite{Chv83}, 
the metric dimension  of $ K_b \times K_c $  
is $ \left \lfloor \frac{2}{3}(b+c-1) \right \rfloor$ if $2b \ge c$
and $c-1$ otherwise~\cite{CHMPPSW07}, and as a special case it is 
$ \left\lceil (4a-1)/3-1 \right\rceil $ for $K_a \times K_a $~\cite{Jag16},
and finally, the metric dimension  of $ K_a \times K_a \times K_a$ is 
is $ \left \lfloor \frac{3}{2} a\right \rfloor$ for $ a \ge 2 $.


%
We make our investigation into the metric dimension of $K_a \times K_b \times K_c$' by introducing a variant of Mastermind, which is 
a famous board game invented by Meirowitz in 1970.
The original game is played by two players, the \emph{codemaker} and the
\emph{codebreaker}, where the codemaker chooses a secret
code consisting of $4$ pegs and $6$ possible colors for each peg and where the 
codebreaker must discover this code by making a sequence of guesses (or questions) 
from the set of possible secrets
until the correct secret  has been found. In each answer, the codemaker gives 
black and white pegs, one black peg for each peg of the
question which is correct in both position and color, and
one white peg for each peg which is correct only in color. 
The problem we consider extends Mastermind in the following ways.
\begin{itemize}
\item We generalize the number of pegs from $4$ to $p$.
\item We generalize the number of pegs from $6$ to $c$.
\item The codemaker gives only black pegs as answers (and thus no longer white pegs, i.e., any information regarding correct colors in incorrect positions).
\item The game is static, i.e., the codebreaker gives all questions already 
at the beginning
of the game, waits for all answers and has a final question
to give the (correct) question.
\item Each peg has its own set of colors $a$, $b$, $c$ 
for given positive integers $a$, $b$, $c$.
\end{itemize}

Almost all such extensions of Mastermind 
have already been investigated (see for example~\cite{ER63,JP09,DDST16} 
for Mastermind with $p$ pegs and $c$ colors, 
\cite{Goo09,JP11} for Black-Peg Mastermind, 
\cite{JD20,Jag16,DDST16,God03} for Static Mastermind). 

The only novel extension is the last one,
i.e., the extension of each peg to an own color set.
More concretely, we generalize Mastermind to Static Black-Peg $(a,b,c)$-Mastermind
or $(a,b,c)$-Mastermind for short, and define it as follows.
We have two players, the \emph{codemaker} and the \emph{codebreaker}.
The codemaker chooses a secret code by assigning a color to each peg, 
chosen from a set of $a$ colors for the first peg, $b$ colors for the second peg and 
$c$ colors for the third peg. 
The codebreaker must discover this code by making a sequence 
of questions of possible secrets at the beginning of the game,
where the set of secrets is the same as the set of questions.
Each answer of the codemaker consists of black pegs, 
one black peg for each peg of the
question which is correct in both position and color.
The codebreaker then receives all answers
and finally has to reply with the (not any more) secret code. 
The goal is to find the a sequence of questions with minimum cardinality 
so that the codebreaker knows the secret from the received answers.
We call this an optimal strategy
Note that it can be assumed w.l.o.g that $a \le b \le c$. To simplify the study, 
we will suppose it to be the case through this paper, unless stated otherwise.

As explained in~\cite{JD20} for the particular case of $(a,a,a)$-Mastermind, 
the cardinality of an optimal strategy 
is equal to the metric dimension of  $ K_a \times K_b \times K_c$. 
To obtain our result, we extend methods from~\cite{JD20}.
We find that optimal strategies mostly consist of questions of the following kind: on one peg
the corresponding color occurs only once over all questions, 
and on the other two pegs the colors occur exactly twice.
As result, we obtain different values for the four main cases 
$3a>b+c$, $3a=b+c$, $3a<b+c$ and $2b \le c$, $3a<b+c$ and $2b> c$.
Our results are summarized in Table~\ref{sum_tab1}


\begin{table}[ht]
    \centering
    \begin{tabular}{c}
    \begin{tikzpicture}
        \draw (-0.25,0) rectangle (12,-7);
        \draw (-0.25,-3) -- (12,-3);
        \draw (-0.25,-5) -- (12,-5);
        \draw (2.5,0) -- (2.5,-7);
        
        \draw (2.5,-1) -- (12,-1);
        \draw (2.5,-2) -- (12,-2);
        \draw (2.5,-4) -- (12,-4);
        \draw (2.5,-6) -- (12,-6);

        \node[anchor = north west] at (0,0){$3a > b+c$};
        \node[anchor = north] at (7,0){$ \left \lfloor \frac{a+b+c}{2} \right \rfloor -1  \leq  f(a,b,c) \leq \left \lfloor \frac{a+b+c}{2} \right \rfloor$};
        \node[anchor = north] at (7,-1){$ f(a,a,2a-1) = 2a -1 = \lfloor \frac{a+a+2a-1}{2}\rfloor $};
        \node[anchor = north] at (7,-2){$ f(a,a,a) = \lfloor \frac{3a}{2} \rfloor$ (result from \cite{JD20})};
        
        \node[anchor = north west] at (0,-3){$3a = b+c$};
        \node[anchor = north] at (7,-3){$(a,b,c) \neq (4,6,6) \rightarrow f(a,b,c) = \left \lfloor \frac{a+b+c}{2} \right \rfloor -1$};
        \node[anchor = north] at (7,-4){$f(4,6,6) = 8 = \left \lfloor \frac{4+6+6}{2} \right \rfloor $};
        
        \node[anchor = north west] at (0,-5){$3a  < b+c$};
        \node[anchor = north] at (7,-5){$2b \leq c \rightarrow f(a,b,c) = c-1$};
        \node[anchor = north] at (7,-6){$2b > c \rightarrow f(a,b,c) = \left \lfloor \frac{2}{3}(b+c-1) \right \rfloor$};
        
    \end{tikzpicture}
    \end{tabular}

    \caption{Summary of the results appearing in this paper.}
    \label{sum_tab1}
\end{table}

The outline of this work is as follows. We start with some notations
and preliminaries in Section~\ref{sec_prel}.
After providing some key lemmas in Section~\ref{sec_key_lemmas},
we determine the metric dimension of $K_a \times K_b \times K_c$
in the following sections and prove the results of Table~\ref{sum_tab1}.
We provide a general lower bound for the metric dimension
$ \lfloor \frac{a+b+c}{2} \rfloor -1 $
in Section~\ref{sec_lower}.
We show an upper bound for the metric dimension
$ \lfloor \frac{a+b+c}{2} \rfloor $
for the case $3a\ge b+c$ in Section~\ref{sec_ge}.
In Section~\ref{sec_aa2a}, as subcases of $3a \ge b+c$, 
we consider the cases 
of $(a,b,c) = (a,a,2a)$ and of $(a,b,c)=(a,a,2a-1)$ and show
that the metric dimension is $2a-1$.
In Section~\ref{sec_eq} we show that for the case $3a=b+c$,
the lower bound  $ \lfloor \frac{a+b+c}{2} \rfloor -1 $ is attained,
except for $(a,b,c)=(4,6,6)$.
In Section~\ref{sec_le} we show that for the case $3a<b+c$ and $2b \le c$,
the metric dimension is $c-1$
and that for the case $3a<b+c$ and $2b > c$,
the metric dimension is $ \left \lfloor \frac{2}{3}(b+c-1) \right \rfloor$.
Finally, in Section~\ref{sec_future},
we discuss our results and the remaining open case, and 
as possible future work we suggest to prove three conjectures.

\section{Notations and Preliminaries}
\label{sec_prel}

Let $\NN$ be the set of positive integers,
and for $n\in\NN$ let $[n]$ be the set $\{1,2,\dots,n\}$. 
We aim to determine the metric dimension 
of $ K_a \times K_b \times K_c$, which we denote by $ f(a,b,c) $ for
$a,b,c \in \NN $.

We define a \emph{strategy} for $(a,b,c)$-Mastermind to be a set of questions.
For a secret $ s=(s_1,s_2,s_3)$ and a question $q=(q_1,q_2,q_3) $, 
let $ g(s,q) = i$ for $ i \in \{0,1,2,3\}$, if the question~$q$ receives~$i$ black pegs for the secret~$s$. 
A strategy $Q$ is \emph{feasible} if for all secrets $s,s'$ with $s\neq s'$ there is 
a question $q\in Q$ such that $g(s,q)\neq g(s',q)$. 
Thus, a \emph{feasible strategy} is one which lets the codebreaker 
win no matter which secret the codemaker chooses, 
An \emph{optimal} strategy is one which is minimal in size.

We make use of the following terminology from \cite{JD20}.

    \begin{definition}\label{def-neighboring2}
    \begin{enumerate}
    \item[{\bf (a)}]
    Two questions $(d_1,d_2,d_3)$ and $(e_1,e_2,e_3)$ are  
    \emph{neighboring} if there is exactly one index $i\in[3]$ such that $d_i=e_i$. 
    \item[{\bf (b)}]
    Two questions $(d_1,d_2,d_3)$ and $(e_1,e_2,e_3)$ are  
    \emph{double-neighboring}, 
    if there are exactly two indices $i\in[3]$ such that $d_i=e_i$. 
    \item[{\bf (c)}] For a given strategy and $ f_1,f_2,f_3  \in \NN $,
    a question $q$ is a \emph{$(f_1,f_2,f_3)$-question},
    if for $i\in[3]$, the $i$-th color of $q$ occurs 
    $f_i$ times on the $i$-th peg (throughout the entire strategy).
    We extend this by allowing $a_i$ to be $\star$, meaning that
    it is not relevant how often the $i$-th color
    occurs on the $i$-th peg.
    \end{enumerate}
    \end{definition}
    
    \begin{definition}
    Consider a strategy for $(a,b,c)$-Mastermind. 
    The \emph{strategy graph} $ G = (V,E) $ of this strategy 
    is defined as the undirected 
    graph whose vertex set $V$ is the set of all 
    questions, and in which two vertices are adjacent if and only if they 
    are neighboring questions. If two questions are double-neighboring, then the strategy graph has multiple edges.
    \end{definition}
    
    In the following we often consider strategies 
    consisting of the following types
    of questions, each of which is named by a letter:
    \begin{center}\begin{tabular}{@{}l@{ }l@{ }l@{\qquad}l@{ }l@{ }l@{}}
     $(1,2,2)$-questions& $\rightarrow $& type $A$&$(2,1,1)$-questions& $\rightarrow $ &type $D$\\
     $(2,1,2)$-questions& $\rightarrow $& type $B$&$(1,2,1)$-questions& $\rightarrow $ &type $E$\\
     $(2,2,1)$-questions& $\rightarrow $& type $C$&$(1,1,2)$-questions& $\rightarrow $ &type $F$\\
    \end{tabular}
    \end{center}
    
    In the corresponding strategy graph each vertex is of degree $1$ or $2$.
    Thus, the graph consists of connected components each of which is either a path or a cycle. 
    
    In the following, we list the questions in a given strategy 
    as follows.
    We order the blocks in an arbitrary way.
    We start with the questions of one block, continue with
    the questions of another block until the last block.
    Each path starts with one of its extreme vertices while each cycle starts
    with an arbitrary vertex. Thus, in one block, the next question 
    is one such that the corresponding vertices are adjacent, i.e., 
    the questions are neighboring.

\section{Key lemmas}
\label{sec_key_lemmas}
    In this secion we present several lemmas which are needed for the feasibility and optinmality proofs later.
    We start with the following lemma~\cite[Lemma 3, Lemma 4(a),(b)]{JD20}. Its proof from~\cite{JD20} holds 
    for the case $a=b=c$, but it works just as well without this assumption.
    
    \begin{lemma}
    \label{lem-forbidden}
    Let $a, b, c \in \NN $. Every feasible strategy for $(a,b,c)$-Mastermind
    has the following properties:
    
    \begin{enumerate}[label=(\alph*)]
    \item
    \label{lem-forbiddena}
    There exists at most
    one color which does not occur on the first peg of any
    question.
    \\
    Analogous statements hold for the second peg
    and for the third peg.
    \item
    \label{lem-forbiddenb}
    There is at most
    one $(1,1,\star)$-question.
    \\
    Analogous statements hold for the cases of $(1,\star,1)$-questions
    and $(\star,1,1)$-questions.
    \\
    In particular,
    setting $ \star = 1 $, there can be at most
    one $(1,1,1)$-question.
    
    \item
    \label{lem-forbiddenc}
    The strategy does not contain three distinct questions one of which is
    a $(1,1,\star)$-question, one of with is a $(1,\star,1)$-question,
    and the third one of which is a $(\star,1,1)$-question.

    \item
    \label{lem-forbiddend}
    Assume that there is a color
    which does not occur on the first peg
    and a color which does not occur on the second peg.
    Then the strategy cannot contain
    a $(1,1,\star)$-question.
    \\
    Analogous statements hold for the first and the third peg,
    and for the second and the third peg.
    
    \item
    \label{lem-forbiddene}
    Assume that there is a color 
    which does not occur on the first peg
    and a color which does not occur on the second peg.
    Then the strategy cannot contain both a $(1,\star,1)$-question and 
    a $(\star,1,1)$-question. 
    \\
    Analogous statements hold for the first and the third peg
    (where the strategy cannot contain both a $(1,1,\star)$-question and
    a $(\star,1,1)$-question),
    and for the second and the third peg
    (where the strategy cannot contain both a $(1,1,\star)$-question and
    a $(1,\star,1)$-question).

    \item
    \label{lem-forbiddenf}
    If the strategy contains two  
    double-neighboring $ (2,2,1) $-questions, then it contains no 
    $ (1,1,\star) $-question, and it holds that on at least one  
    of the first two pegs all colors occur.
    \\
    Analogous questions hold for double-neighboring $ (2,1,2) $-questions
    and for two double-neighboring $ (1,2,2) $-questions.
    \item
    \label{lem-forbiddeng}
    Assume that the strategy contains a $ (2,2,1) $-question,
    a $ (2,\star,1) $-question and
    a $ (\star,2,1) $-question, 
    where the first and the second question 
    are neighboring on the first peg, 
    and the first and the third question are neighboring on the second peg. 
    Then the strategy contains no 
    $ (1,1,\star) $-question, and it holds that on at least one  
    of the first two pegs all colors occur in the questions.
    \\
    Analogous statements hold for a $ (2,1,2) $-question, a $ (\star,1,2) $-question and
    a $ (2,1,\star) $-question, 
    and for a $ (1,2,2) $-question, a $ (1,\star,2) $-question and
    a $ (1,2,\star) $-question. 
    
    \end{enumerate}
    \end{lemma}
    
\begin{corollary}
\label{cor_c-1}
Let $ a,b,c \in \NN$ with $ a \le b \le c$. Then it holds that $ f(a,b,c) \ge c-1$.
\end{corollary}
\begin{proof}
This follows directly from Lemma~\ref{lem-forbidden}\ref{lem-forbiddena}, as having less questions would lead to at least two missing colors in the last peg.    
\end{proof}
    
    The next lemma shows that an increase of the colors on one peg increases
    the metric dimension not more than by $1$.
    
    \begin{lemma}
    \label{lem_increase}
     Let $a, b, c \in \NN $. Then it holds that 
     $f(a,b,c+1) \leq f(a,b,c) +1$, $f(a,b+1,c) \leq f(a,b,c) +1$, and
     $f(a+1,b,c) \leq f(a,b,c) +1$.
    \end{lemma}
    
    \begin{proof}
    
    We first show that $f(a,b,c+1) \leq f(a,b,c) +1$.
    Let $Q$ be an optimal strategy for $(a,b,c)$-Mastermind, i.e., it has $f(a,b,c)$ questions.
    The statement holds for $a=b=c=1$ because, obviously, $f(1,1,1)=0$ and $f(1,1,2)=1$. Thus, 
    assume in the following that $\max\{a,b,c\}\ge 2$ and thus $|Q|\neq\emptyset$.
    Let $(q_1,q_2,q_3) \in Q$ be arbitrary and define $Q' = Q \cup \{(q_1,q_2,c+1)\}$. 
    We show that $Q'$ is a feasible strategy for $(a,b,c+1)$-Mastermind by showing that 
    all pairs of secrets $s=(x,y,z)$ and $s'=(x',y',z')$ are distinguished by at least one question of $Q'$.
    
    \begin{itemize}
    \item $z,z' \neq c+1$.
    
    The secrets $s$ and $s'$ are secrets of $(a,b,c)$-Mastermind and therefore must be 
    distinguished by at least one question $q \in Q\subseteq Q'$. Hence, the statement is verified.
    
    \item $z=z'=c+1$.
    
    Since $Q$ is feasible, there exists a question $q=(q'_1,q'_2,q'_3)$ of $Q$ that 
    distinguishes $(x,y,c)$ and $(x',y',c)$. 
    
We consider two subcases.
\begin{itemize}
\item $q'_3 \in [c-1]$. 

Then $g(q,(x,y,c)) = g(q,s)$ and $g(q,(x',y',c)) = g(q,s')$. 

\item $q'_3 = c$. 

Then $g(q,(x,y,c)) = 1 + g(q,s)$ and $g(q,(x',y',c)) = 1 + g(q,s')$. 
\end{itemize}

In both cases, $g(q,(x,y,c)) - g(q,(x',y',c)) = g(q,s) - g(q,s')$ and since $q$ 
distinguishes $(x,y,c)$ and $(x',y',c)$, it also distinguishes $s$ and $s'$.

\item $z= c+1$ and $z' \neq c+1$.

If $(q_1,q_2,c+1)$ distinguishes $s$ and $s'$ then there is nothing to show. Thus, assume that $g((q_1,q_2,c+1),s) = g((q_1,q_2,c+1),s') = X$ for some $X \in \{0,1,2,3\}$. We have $g((q_1,q_2,q_3),s)= X-1$ 
and $g((q_1,q_2,q_3),s') \in \{X,X+1\} $, proving that $(q_1,q_2,q_3) \in Q'$ distinguishes $s$ and $s'$.

\item $z \neq c+1$ and $z' = c+1$.

This case is analogous to the case 
$z= c+1$ and $z' \neq c+1$.
\end{itemize}

In conclusion, all possible pairs of secrets of $(a,b,c+1)$-Mastermind are indeed distinguished by $Q'$. 
Since this reasoning holds for all feasible strategies $Q$, 
$f(a,b,c+1) \leq f(a,b,c) +1$ follows.

The proofs of the statements  $f(a,b+1,c) \leq f(a,b,c) +1$ and
$f(a+1,b,c) \leq f(a,b,c) +1$ are entirely analogous.
\end{proof}

\begin{remark}
\label{rem_increase}
Given a feasible strategy $Q$ for $(a,b,c)$-Mastermind,
the proof of Lemma~\ref{lem_increase} gives a construction 
for a feasible strategy for $(a,b,c+1)$-Mastermind as follows.
Choose an arbitrary question $(q_1,q_2,q_3) \in Q$. Then
add the question $(q_1,q_2,c+1)$ to $Q$ and obtain the strategy
$Q'$ for $(a,b,c+1)$-Mastermind.

Analogously, the proof also gives constructions 
for a feasible strategy for $(a,b+1,c)$-Mastermind and for a feasible strategy
for $(a+1,b,c)$-Mastermind. 
\end{remark}

Now let us present a way to decrease the value of the number of colors of one peg 
so that we still have a feasible strategy. For this, it seems quite intuitive that we always have $f(a,b,c) \le f(a+1,b,c)$ (and similarly for the second and third peg). To prove this result, one could show how to transform any strategy 
for $(a+1,b,c)$-Mastermind into a strategy for $(a,b,c)$-Mastermind using the 
same number of questions. One way to do this would be to "change" one color in the first peg 
to another color. 
However, this does not always lead to a feasible strategy. For seeing this, consider
Table~\ref{tab_projection}, where (a) 
gives a strategy for $(3,3,3)$-Mastermind with $4$ questions and (b) the corresponding
strategy for $(2,3,3)$-Mastermind after changing the $3$ of the first peg of $q_4$ into a $2$. 
However, this does not yield a feasible strategy, as the 
secrets $(2,1,2)$ and $(1,3,3)$ are not distinguished.

\begin{table}[t]
    \centering
	\setbox9=\hbox{\begin{tabular}{c || c | c | c}
                Peg & $1$ & $2$ & $3$
                \\
                \hline
                \hline
                $q_1$ & $1$ & $1$ & $1$
                \\
                \hline
                $q_2$ & $1$ & $2$ & $2$
                \\
                \hline
                $q_3$ & $2$ & $2$ & $3$
                \\
                \hline
                $q_4$ & $2$ & $3$ & $1$
                \end{tabular}}
    \subcaptionbox[table]{Feasible strategy for $(a,b,c)$-Mastermind.}{
    		\raisebox{\dimexpr\ht9-\height}[\ht9][\dp9]{
                \begin{tabular}{c || c | c | c}
                Peg & $1$ & $2$ & $3$ 
                \\
                \hline
                \hline
                $q_1$ & $1$ & $1$ & $1$ 
                \\
                \hline
                $q_2$ & $1$ & $2$ & $2$ 
                \\
                \hline
                $q_3$ & $2$ & $2$ & $3$ 
                \\
                \hline
                $q_4$ & $3$ & $3$ & $1$
                \end{tabular}}}
    \hspace*{2cm}
	\subcaptionbox[table]{Non-feasible strategy for $(2,3,3)$-Mastermind 
 obtained from the strategy of (a).}{
                \box9}
    \caption{Strategies for $(3,3,3)$-Mastermind and $(2,3,3)$-Mastermind.} 
    \label{tab_projection}
\end{table}

In the next lemma, we provide sufficient conditions to transform a feasible strategy for $(a+1,b,c)$-Mastermind into a feasible strategy for $(a,b,c)$-Mastermind using the same number of questions.
To do so, we first need the following definition, formalizing the notion of "changing a color" used earlier.

\begin{definition}
\label{def_proj}
    Let $a \ge2$, let $Q$ be a feasible strategy for $(a,b,c)$-Master\-mind, 
    and let $i,j\in[a]$ be two distinct colors of the first peg. 
    Then we obtain a strategy $Q_1(i,j)$ for $(a-1,b,c)$-Mastermind
    by replacing each occurrence of $j$ with $i$ 
    and then decreasing each color on the first peg that is greater than~$j$ by~$1$.
    We call this strategy $Q_1(i,j)$ the \emph{projection strategy} from color $j$ to $i$
    for the first peg. 
    We define the projection strategies $Q_2(i,j)$ and $Q_3(i,j)$ for the second and third peg 
    analogously (where $b \ge 2$, $i,j\in[b]$ and $ c \ge 2 $, $i,j\in[c]$, respectively).
\end{definition}

\begin{lemma}
\label{lem_proj}
    Let $a \ge2$, let $Q$ be a feasible strategy for $(a,b,c)$-Mastermind,
    and consider distinct colors $i,j\in[a]$ of the first peg. 
    If for all $y\in[b]$ and $z\in[c]$ there exists a question $(q_1,q_2,q_3) \in Q$ 
    such that $q_1 \in \{i,j\}$,  $q_2 \neq y$ and $q_3 \neq z$,
    then the projection strategy $Q_1(i,j)$ is feasible for $(a-1,b,c)$-Mastermind.
    
    Analogous statements hold for the second peg and $(a,b-1,c)$-Mastermind, and for the third peg 
    with $(a,b,c-1)$-Mastermind.
\end{lemma}

\begin{proof}
W.l.o.g., consider the projection $Q_1(a-1,a)$, as the projection principle is 
the same for every peg, and by changing the color names, every projection for the first peg 
is analogous to a $Q_1(a-1,a)$ projection.

In the following we show that each pair of secrets $s=(x,y,z)$ and $s'=(x',y',z')$ of $(a-1,b,c)$-Mastermind 
is distinguished by at least one question of $Q_1(a-1,a)$.
As $s$ and $s'$ are also secrets of $(a,b,c)$-Mastermind, there is a 
question $q=(q_1,q_2,q_3)$ of $Q$ that distinguishes these secrets. 

If $q_1 \neq a$, then $q\in Q_1(a-1,a)$ distinguishes $s$ and $s'$. Assume therefore in the following  that $q_1 = a$. Then $q$ has been replaced by $q' = (a-1,b,c)$ in $Q_1(a-1,a)$.
We have the following four cases.

\begin{itemize}
\item $x,x' \neq a-1$. 

Then $g(q,s) = g(q',s)$ and $g(q,s') = g(q',s')$, so $s$ and $s'$ are still distinguished. 

\pagebreak[3]

\item $x=x'=a-1$. 

Then $g(q,s) = g(q',s)-1$ and $g(q,s') = g(q',s')-1$, and the difference in the number of black pegs 
remains the same. Therefore, $q'$ distinguishes $s$ and $s'$. 

\item $x=a-1$ and $x' \neq a-1$. 

Then by assumption there is a question $(q_1',q_2',q_3') \in Q$ such that $q_1' \in \{a-1,a\}$, $q_2' \neq y'$ and 
$q_3' \neq z'$. Therefore, the question $(a-1,q_2',q_3')$ is in $Q_1(a-1,a)$. 
For this question it holds that $g((a-1,q_2',q_3'),s) \ge 1$ and $g((a-1,q_2',q_3'),s') = 0$. 
Therefore, it distinguishes $s$ and $s'$. 

\item $x \neq a-1$ and $x' = a-1$. 

This case is analogous to the case 
$x=a-1$ and $x' \neq a-1$. 

\end{itemize}

In conclusion, each pair of secrets of $(a-1,b,c)$-Mastermind is distinguished by $Q_1(a-1,a)$.

The proofs of the two other projection strategies work analogously.
\end{proof}

\begin{corollary}
\label{cor_proj}
If there exists an optimal strategy $Q$ and colors $i,j$ which fulfill the conditions 
of Lemma~\ref{lem_proj}, then it holds that
$f(a-1,b,c) \le f(a,b,c)$ (and analogous for the second and the third peg).
\end{corollary}

In the remainder of this paper, we will introduce several strategies using
questions of type A, B, C.
We will use the following lemma to prove the feasibility of these strategies.
To do so, we provide a list of patterns such that strategies avoiding all these
patterns are feasible.

Let $T \in \{A,B,C\}$ be a type of question. The \emph{single peg} of $T$ is the peg that contains the color occurring only in the question of type $T$,
and the two other pegs are the \emph{double pegs} of $T$. 
For example, the single peg of a question of type $A$ is the first peg and the other 
two pegs are its double pegs. 

We say that a set of $k$ questions is \emph{a sequence of 
type $T_1T_2 \ldots T_k$} if these questions can be arranged 
in a sequence $q_1,q_2,\dots,q_k$ such that
\begin{description}
\item[(a)] $q_i$ and $q_{i+1}$ are neighboring to each other for all $i\in[k-1]$, and
\item[(b)] $q_i$ is of type $T_i\in\{A,B,C\}$ for all $i\in[k]$.
\end{description}

\begin{lemma}
\label{lem:super_lemma}
    Let a set $Q$ of questions contain only questions of type $A$, $B$ and $C$ and avoid the following patterns: 
    \renewcommand{\labelenumi}{\theenumi)}
    \begin{description}
        \item[(1)] two colors missing on the same peg,
        \item[(2)] a missing color on each peg,
        \item[(3)] a sequence of type $TTT$ and missing colors on both of the double pegs of $T$.
        \item[(4)] a sequence of type $TTT$ and a pair of double-neighboring $TT$,
        \item[(5)] a sequence of type $T_1T_1T_2T_2$ ($T_1 \neq T_2$) with a missing color on the peg that is double for both $T_1$ and $T_2$,
        \item[(6)] a sequence of type $T_1T_1T_2T_3T_3$, where all three types distinct,
        \item[(7)] two different sequences of the same type $T_1T_2$ ($T_1 \neq T_2$),
        \item[(8)] a sequence of type $T_1T_2$ ($T_1 \neq T_2$), a sequence of type $T_2T_2$ of double-neighboring questions, and a missing color on the single peg of $T_2$,
        \item[(9)] a sequence of type $TT$ of double-neighboring questions, with missing colors on both the double pegs of $T$,
        \item[(10)] two different sequences of type $TT$ of double-neighboring questions,
        \item[(11)] two different sequences of types $T_1T_1$ and $T_2T_2$, both of these consisting of double-neighboring questions, and missing colors on the single pegs of both $T_1$ and $T_2$, and
        \item[(12)] sequences of types $AA$, $BB$, and $CC$ of double-neighboring questions.
    \end{description}
    Then $Q$ is a feasible strategy.
\end{lemma}

\begin{proof}
Let $Q$ be a strategy which consists only of questions of type $A$, $B$ or $C$ 
and which avoids all the listed patterns. 
Furthermore, let $s_1 = (x_1,y_1,z_1)$ and $s_2 = (x_2,y_2,z_2)$ be two secrets with $s_1 \neq s_2$. 
We show that $Q$ distinguishes $s_1$ and $s_2$.

First, let us prove a claim that will simplify the proof.

\begin{claim}
Let $x_1 \neq x_2$, $y_1 \neq y_2$, $z_1 \neq z_2$.
\begin{description}
\item[(a)]
Then the following are the only questions which do not distinguish $s_1$ and $s_2$:
\begin{description}
\item[(i)] questions $q$ with $g(q,s_1)=g(q,s_2)=0$, i.e., 
which contain neither $x_1$ nor $x_2$ on the first peg, 
neither $y_1$ nor $y_2$ on the second peg, 
neither $z_1$ nor $z_2$ on the third peg, ]
\item[(ii)] questions $ (x_i,y_j,\star)$, $(x_i,\star,z_j)$, $(\star,y_i,z_j)$ for $ i,j \in [2]$, $ i \neq j$,
where ``$\star$'' stands for a color different from the
colors of the two secrets on this peg. 
\end{description}

\item[(b)]
In particular, $ (x_i,y_j,z_k)$ for $ i,j,k \in [2]$ distinguishes $s_1$ and $s_2$.

\item[(c)]
In particular, a question distinguishes $s_1$ and $s_2$, if the color on one peg is
the color of $s_1$ (resp. $s_2$), and the colors on the two other pegs
differ from the colors of $s_2$ (resp. $s_1$).
\end{description}
\end{claim}

\begin{proof}[Proof (Claim).]
\begin{description}
\item[(a)]
As $x_1 \neq x_2$, $y_1 \neq y_2$, $z_1 \neq z_2$ hold, $g(q,s_1)=g(q,s_2)>0$ implies $g(q,s_1)=g(q,s_2)=1$. 
Then the color on one peg of this question is only equal 
to the color of $s_1$, 
the color on another peg is only equal 
to the color of $s_2$. 
and the color on the remaining peg 
is different from both colors of $s_1$ and $s_2$ on this peg,
\item[(b)] This follows, as these questions are not listed in (a).
\item[(c)] This follows, as these questions are not listed in (a).
\end{description}
\let\qed\relax\end{proof}

\bigskip

Since $s_1 \neq s_2$ holds, $s_1$ and $s_2$ differ in at least one peg. 
W.l.o.g., let $ x_1 \neq x_2 $. As $Q$ avoids pattern 1, there is a question
$q_1$ containing $x_1$ or $x_2$ on the first peg, say $ q_1 = (x_1,\star,\star)$. 

If $q_1$ contains neither $y_2$ on the second nor $z_2$ on the third peg, then it distinguishes $s_1$ and $s_2$. 
Thus, we can assume w.l.o.g. that $q_1$ contains $y_2$ on the second peg. 
Moreover, we can also assume that $y_1 \neq y_2$. This is because $y_1=y_2$ implies that $q_1$ contains $z_2$ on the third peg and $z_1\neq z_2$ and we could let the third peg take this role.
In summary, assume in the following that $Q$ contains a question $q_1 = (x_1,y_2, \star)$ with 
$x_1 \neq x_2$ and  $y_1 \neq y_2$. 

We consider different cases, depending on the type of $q_1$. 
To follow the proof more easily, Figure~\ref{fig:cases_super_lemma} shows the state of the questions at the start of each subcase of the proof. This information is also repeated in the margin at the beginning of each subcase.

\begin{figure}[H]
    \resizebox{\linewidth}{!}{
    \begin{tikzpicture}
    
        \node[draw, align = left] (c1) at (0,0) {
        \textbf{Case 1}\\
        $q_1\colon (x_1,y_2,\star)$ - $A$
        };

        \node[draw, align = left] (c11) at (-3.75,-2) {
        \textbf{Case 1.1}\\
        $q_1\colon (x_1,y_2,\star)$ - $A$\\
        $q_2\colon (\star,y_2,z_1)$ - $A$
        };
        \draw[->] (c1) -- (c11);

        \node[draw, align = left] (c111) at (-6,-4.5) {
        \textbf{Case 1.1.1}\\
        $q_1\colon (x_1,y_2,\star)$ - $A$\\
        $q_2\colon (\star,y_2,z_1)$ - $A$\\
        $q_3\colon (x_2,\star,z_1)$ - $A$
        };
        \draw[->] (c11) -- (c111);
        
        \node[draw, align = left] (c112) at (-2.5,-4.5) {
        \textbf{Case 1.1.2}\\
        $q_1\colon (x_1,y_2,\star)$ - $A$\\
        $q_2\colon (\star,y_2,z_1)$ - $A$\\
        $q_3\colon (x_2,\star,z_1)$ - $B$
        };
        \draw[->] (c11) -- (c112);

        \node[draw, align = left] (c12) at (4.5,-2) {
        \textbf{Case 1.2}\\
        $q_1\colon (x_1,y_2,\star)$ - $A$\\
        $q_2\colon (\star,y_2,z_1)$ - $C$
        };
        \draw[->] (c1) -- (c12);        

        \node[draw, align = left] (c121) at (1,-4.5) {
        \textbf{Case 1.2.1}\\
        $q_1\colon (x_1,y_2,\star)$ - $A$\\
        $q_2\colon (\star,y_2,z_1)$ - $C$\\
        $q_3\colon (x_2,y_1,\star)$ - $A$
        };
        \draw[->] (c12) -- (c121);
        
        \node[draw, align = left] (c122) at (4.5,-4.5) {
        \textbf{Case 1.2.2}\\
        $q_1\colon (x_1,y_2,\star)$ - $A$\\
        $q_2\colon (\star,y_2,z_1)$ - $C$\\
        $q_3\colon (x_2,y_1,\star)$ - $B$
        };
        \draw[->] (c12) -- (c122);

        \node[draw, align = left] (c123) at (8,-4.5) {
        \textbf{Case 1.2.3}\\
        $q_1\colon (x_1,y_2,\star)$ - $A$\\
        $q_2\colon (\star,y_2,z_1)$ - $C$\\
        $q_3\colon (x_2,y_1,\star)$ - $C$
        };
        \draw[->] (c12) -- (c123);
        
        \node[draw, align = left] (c1121) at (-4.5,-7.5) {
        \textbf{Case 1.1.2.1}\\
        $q_1\colon (x_1,y_2,\star)$ - $A$\\
        $q_2\colon (\star,y_2,z_1)$ - $A$\\
        $q_3\colon (x_2,\star,z_1)$ - $B$ \\
        $q_4\colon (x_2,y_1,\star)$ - $B$ 
        };
        \draw[->] (c112) -- (c1121);

        \node[draw, align = left] (c1122) at (-0.5,-7.5) {
        \textbf{Case 1.1.2.2}\\
        $q_1\colon (x_1,y_2,\star)$ - $A$\\
        $q_2\colon (\star,y_2,z_1)$ - $A$\\
        $q_3\colon (x_2,\star,z_1)$ - $B$ \\
        $q_4\colon (x_2,y_1,\star)$ - $C$ 
        };
        \draw[->] (c112) -- (c1122);

        \node[draw, align = left] (c2) at (0,-10) {
        \textbf{Case 3}\\
        $q_1\colon (x_1,y_2,\star)$ - $C$
        };

        \node[draw, align = left] (c21) at (-4,-12) {
        \textbf{Case 3.1}\\
        $q_1\colon (x_1,y_2,\star)$ - $C$\\
        $q_2\colon (x_1,\star,\star)$\\
        $q_3\colon (\star, y_2, \star)$
        };
        \draw[->] (c2) -- (c21);

        \node[draw, align = left] (c22) at (4,-12) {
        \textbf{Case 3.2}\\
        $q_1\colon (x_1,y_2,\star)$ - $C$\\
        $q_2\colon (x_1,y_2,\star)$ - $C$
        };
        \draw[->] (c2) -- (c22);        

        \node[draw, align = left] (c221) at (2,-14.5) {
        \textbf{Case 3.2.1}\\
        $q_1\colon (x_1,y_2,\star)$ - $C$\\
        $q_2\colon (x_1,y_2,\star)$ - $C$\\
        $q_3\colon (x_2, y_1, \star)$
        };
        \draw[->] (c22) -- (c221);
        
        \node[draw, align = left] (c222) at (6,-14.5) {
        \textbf{Case 3.2.2}\\
        $q_1\colon (x_1,y_2,\star)$ - $C$\\
        $q_2\colon (x_1,y_2,\star)$ - $C$\\
        $q_3\colon (x_2,\star , z_1)$
        };
        \draw[->] (c22) -- (c222);

    \end{tikzpicture}
    }
    \caption{The questions and the colors of the secrets $s_1$ and $s_2$ at the start of each subcase of Lemma~\ref{lem:super_lemma}}
    \label{fig:cases_super_lemma}
\end{figure}
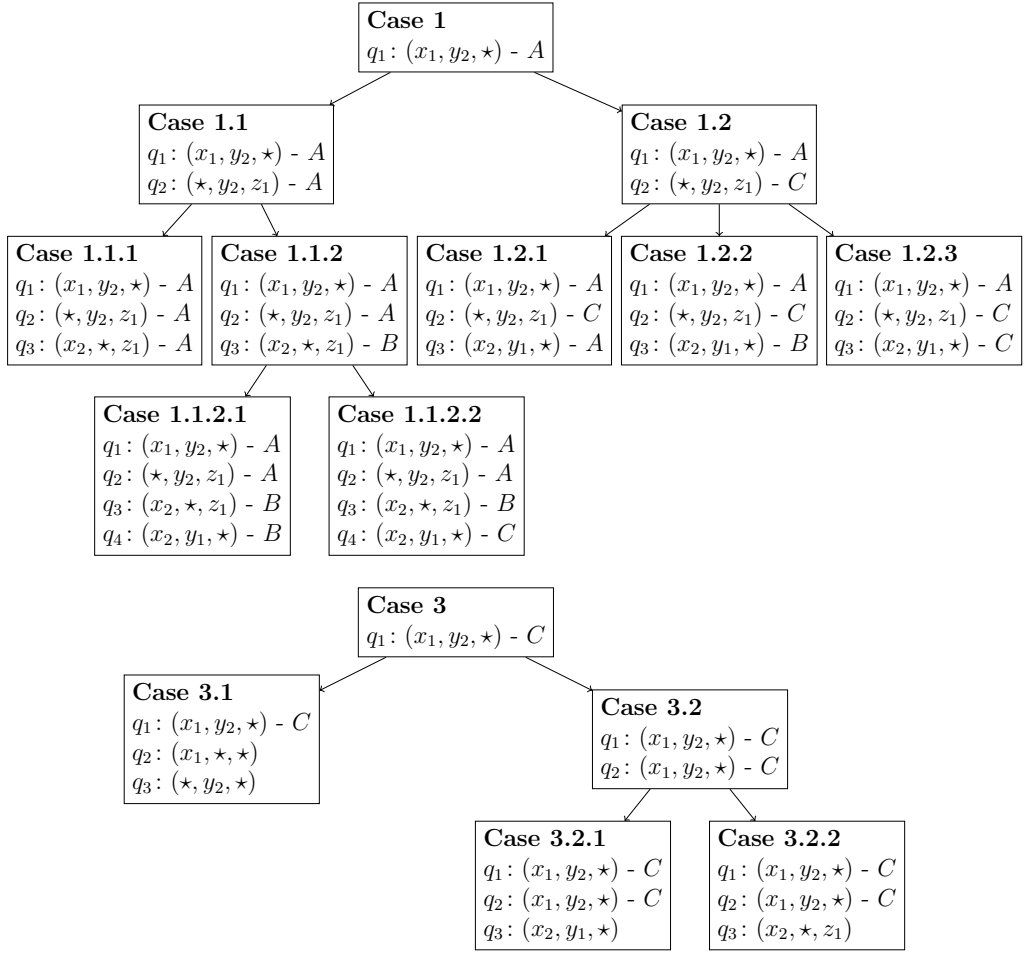

\begin{description}

\case[$q_1\colon (x_1,y_2,\star)$ - $A$]1{$q_1$ is of type $A$}%
Since $q_1$ is of type $A$, no other question contains the color $x_1$ on the first peg and 
there is a question $q_2$, different from $q_1$, which contains $y_2$ on the second peg.
\footnote{In the following, we will omit the reference to the 
peg in question when talking about the colors $x_i,y_i,z_i$ ($i\in[2]$), as these colors always 
refer to the first, second, and third peg, respectively. Thus we write, e.g., that a question 
does not contain $z_2$ when we mean that it does not contain 
that color on the third peg.}
Similarly to $q_1$, if $q_2$ does not contain $z_1$, it distinguishes $s_1$ and $s_2$ (because it also does not contain $x_1$). 
Hence, $q_2=(\star,y_2,z_1)$.
Moreover, we have $z_1 \neq z_2$.
Thus, in Case 1,
$x_1 \neq x_2$, $y_1 \neq y_2$, $z_1 \neq z_2$ holds,
and we can apply the claim.

We now look at the type of $q_2$. Note that it cannot be of type $B$, 
since $q_2$ contains $y_2$ which is also in $q_1$.
\pagebreak[3]

\begin{description}
\case[$q_1\colon (x_1,y_2,\star)$ - $A$\\
  $q_2\colon (\star,y_2,z_1)$ - $A$]{1.1}{$q_2$ is of type $A$}
Since $q_2$ is of type $A$, there is a question $q_3$, 
different from $q_2$, that contains $z_1$. If $q_3 = q_1$ , 
then we have $q_1 = (x_1,y_2,z_1)$ and 
$x_1 \neq x_2,y_1 \neq y_2,z_1 \neq z_2$, therefore $q_1$ distinguishes $s_1$ and $s_2$. We can then assume that $q_1 \neq q_3$.

Similarly to previous cases, $q_3$ must contain either $x_2$ or $y_2$, as otherwise it would distinguish $s_1$ and $s_2$. However, 
since $y_2$ is already in $q_1$ and $q_2$, it cannot occur in $q_3$. Therefore, $q_3 = (x_2, \star, z_1)$ and $q_3$ is either of type $A$
or of type $B$.

\begin{description}
\case[$q_1\colon (x_1,y_2,\star)$ - $A$\\
        $q_2\colon (\star,y_2,z_1)$ - $A$\\
        $q_3\colon (x_2,\star,z_1)$ - $A$]{1.1.1}{$q_3$ is of type $A$}
%
Since $Q$ avoids pattern~3, 
either $y_1$ or $z_2$ occur in some question $q_4$. 
W.l.o.g, we can assume that $q_4 = (\star, y_1, \star)$. 
Note that $q_4 \neq q_3$, as otherwise $q_4=(x_2,y_1,z_1) $, which is not possible by Claim~(b). Further,
$q_4 \notin\{ q_1, q_2\}$ holds as both questions contain $y_2$. Thus, $q_4 \notin\{ q_1, q_2, q_3\}$. 
As $x_2$ occurs in $q_3$, and $q_3$ is of type $A$, Claim~(a) yields
$q_4 = (\star, y_1, z_2)$. 

Since $q_4$ is of type $A$, $B$ or $C$, there exists a question $q_5$ 
that contains $y_1$ or $z_2$. 
Moreover, since $Q$ avoids pattern~4, $q_5$ cannot contain both $y_1$ and $z_2$. 
Since $q_1$ and $q_3$ are of type~$A$, $q_5$ cannot contain $x_1$ or $x_2$ either. 
Therefore, $q_5$ will either contain $y_1$ and neither $x_2$ nor $z_2$, or $q_5$ will 
contain $z_2$ and neither $x_1$ nor $y_1$. By Claim~(c), in both cases, $q_5$ 
distinguishes $s_1$ and $s_2$.

\case[$q_1\colon (x_1,y_2,\star)$ - $A$\\
        $q_2\colon (\star,y_2,z_1)$ - $A$\\
        $q_3\colon (x_2,\star,z_1)$ - $B$]{1.1.2}{$q_3$ is of type $B$}%
As $q_3$ is of type $B$, there exists another question $q_4$ that contains $x_2$, 
and by Claim~(a), it must also contain $y_1$ or $z_1$,
Since $z_1$ already occurs in the questions $q_2$ and $q_3$, $q_4 = (x_2, y_1, \star)$ follows.

\begin{description}
\case[$q_1\colon (x_1,y_2,\star)$ - $A$\\
        $q_2\colon (\star,y_2,z_1)$ - $A$\\
        $q_3\colon (x_2,\star,z_1)$ - $B$\\
        $q_4\colon (x_2, y_1, \star)$ - $B$]{1.1.2.1}{$q_4$ is of type $B$}%
Since $Q$ avoids pattern~5, $z_2$ is not a missing color, and it must occur in a 
question $q_5$. 
Then $q_5$ must be different from $q_2$ and $q_3$, 
which contain $z_1$, and from $q_1$, as otherwise $q_5=(x_1,y_2,z_2)$ would hold, 
and $q_4$, as otherwise $q_5=(x_2,y_1,z_2)$ would hold, each of these cases contradicting Claim~(b). Moreover, since $q_1$ is of type $A$ and 
$q_4$ of type $B$, $q_5$ contains neither $x_1$ nor $y_1$, but contains $z_2$. 
Therefore, by Claim~(c),  $q_5$ distinguishes $s_1$ and $s_2$. 

\case[$q_1\colon (x_1,y_2,\star)$ - $A$\\
        $q_2\colon (\star,y_2,z_1)$ - $A$\\
        $q_3\colon (x_2,\star,z_1)$ - $B$\\
        $q_4\colon (x_2, y_1, \star)$ - $C$]{1.1.2.2}{$q_4$ is of type $C$}
Then there exists another question $q_5$ that contains $y_1$.
Because of $y_1 \neq y_2$,
it is different from $q_1$ and $q_2$. 
As $q_3$ and $q_4$ contain $x_2$,
$q_5$ cannot contain it. 
By Claim~(a), must contain $z_2$, i.e.,
$q_5 = (\star, y_1, z_2)$. 

Since $y_1$ occurs in both $q_4$ and $q_5$, $q_5$ cannot be of type $B$. It cannot be of type $C$ either, as $q_1,q_2,\dots,q_5$ 
would give the sequence $AABCC$ which is 
forbidden by pattern~6. So $q_5$ must be of type $A$. Therefore, there is another 
question $q_6$ that contains~$z_2$. This question cannot contain $x_1$, which only 
occurs in $q_1$, nor $y_1$, which occurs in both $q_4$ and $q_5$. However, it contains
$z_2$. Therefore, by Claim~(c),  $q_5$ distinguishes $s_1$ and $s_2$. 
\end{description}

\end{description}

\case[$q_1\colon (x_1,y_2,\star)$ - $A$\\
        $q_2\colon (\star,y_2,z_1)$ - $C$]{1.2}{$q_2$ is of type $C$}%
Since $Q$ avoids pattern~2, at least one color from $x_2$, $y_1$ or $z_2$ occurs in a question. 
Note that the cases where $x_2$ or $z_2$ occur in a question are symmetric. Moreover, 
if $q_3$ contains $x_2$, it must also contain $y_1$, since $z_1$ occurs in the question $q_2$ of type~$C$.
Analogously, if $q_3$ contains $z_2$, it must also contain $y_1$, since 
$x_1$ occurs in the question $q_1$ of type~$A$.
If $y_1$ occurs in a question, it must also contain $x_2$ or $z_2$. 
Therefore, we can assume w.l.o.g. that there exists a question $q_3 = (x_2,y_1,\star)$.

\begin{description}
\case[$q_1\colon (x_1,y_2,\star)$ - $A$\\
        $q_2\colon (\star,y_2,z_1)$ - $C$\\
        $q_3\colon (x_2,y_1,\star)$ - $A$]{1.2.1}{$q_3$ is of type $A$}%
Since $q_3$ is of type $A$, there is another question $q_4$ that contains $y_1$. 
As the question $q_3$ of type~$A$ contains $x_2$, 
$q_4$ cannot contain $x_2$ and it must contain $z_2$.
$q_4$ question cannot be of type $B$, as $y_1$ occurs in question $q_3$ and it cannot 
be of type $C$ as $Q$ avoids pattern~7 and both pairs of questions $q_1q_2$ and $q_3q_4$ 
would both be a sequence $AC$. Therefore, $q_4$ is of type $A$. Note now that 
if we forget questions $q_1$ and $q_2$, the questions $q_3$ and $q_4$ are 
exactly in the same situation as questions 
$q_1$ and $q_2$ in Case~1.1. 
This case has been concluded with the fact that $s_1$ and $s_2$ are distinguished, 
and thus this is also the case here.

\pagebreak[3]

\case[$q_1\colon (x_1,y_2,\star)$ - $A$\\
        $q_2\colon (\star,y_2,z_1)$ - $C$\\
        $q_3\colon (x_2,y_1,\star)$ - $B$]{1.2.2}{$q_3$ is of type $B$}%
There exists another question $q_4$ that contains $x_2$. It cannot be question $q_2$, as then 
$q_4= (x_2,y_2,z_1) $ would distinguish $s_1$ and $s_2$ by Claim~(b). Moreover, 
since $q_2$ is of type $C$  and $q_3$ of type $B$, $q_4$ can contain neither 
$y_1$ nor $z_1$. 
Therefore, by Claim~(c), it distinguishes $s_1$ and $s_2$.

\pagebreak[3]

\case[$q_1\colon (x_1,y_2,\star)$ - $A$\\
        $q_2\colon (\star,y_2,z_1)$ - $C$\\
        $q_3\colon (x_2,y_1,\star)$ - $C$]{1.2.3}{$q_3$ is of type $C$}%
Since $q_3$ is of type $C$, there exist questions $q_4$ and $q_5$, different from $q_3$, that contain $x_2$ and $y_1$, respectively. 

\begin{description}
\item $q_4 \neq q_5$.

Then $q_4$ can contain neither $y_1$, 
which occurs in both $q_3$ and $q_5$, nor $z_1$, which occurs in the question $q_2$ of type~$C$. 
Therefore, by Claim~(c), it distinguishes $s_1$ and $s_2$.

\item $q_4 = q_5$.

Then we have the sequence $AC$ of the questions $q_1q_2$ 
and the double-neighboring questions $q_3q_4$. Since $Q$ avoids pattern~8, there is no 
missing color on the third peg (the single peg of type~$C$).
Thus, there must be a question $q_6$ that contains $z_2$ and is not equal to $q_2$ since $z_1 \neq z_2$.
The cases where $q_6 \in \{q_1,q_3,q_4\}$ lead to $q_6=(x_1,y_2,z_2)$, $q_6=(x_2,y_1,z_2)$
and $q_6= (x_2,y_1,z_2)$, respectively. By Claim~(b), in all three cases, $q_6$ distinguishes $s_1$ and $s_2$. 
Therefore, $q_6$ is a new question. 
However, $q_6$ can neither contain $x_1$, which occurs on the single peg of $q_1$, nor $y_1$, 
which occurs in both $q_3$ and $q_4$. 
Therefore, by Claim~(c), it distinguishes $s_1$ and $s_2$.
\end{description}
\end{description}
\end{description}

\case[]{2}{$q_1$ is of type $B$}
By symmetry, switching the roles of pegs~$1$ and~$2$, this works analogously to Case 1.

\case[$q_1\colon (x_1,y_2,\star)$ - $C$]{3}{$q_1$ is of type $C$}%
Then there exists another question $q_2$ that contains $x_1$ and another question $q_3$ 
that contains $y_2$. These two questions can be the same, in the case that $q_2=q_3$ is double 
neighboring to $q_1$. Let us now study two cases depending on whether they are 
different questions or not.

\begin{description}
\case[$q_1\colon (x_1,y_2,\star)$ - $C$\\
        $q_2\colon (x_1,\star,\star)$\\
        $q_3\colon (\star, y_2, \star)$]{3.1}{$q_2 \neq q_3$}
To avoid distinguishing $s_1$ and $s_2$, $q_2$ must contain $z_2$ and $q_3$ must 
contain $z_1$, with $z_1 \neq z_2$. 
Thus, for the remainder of Case 3.1 we can assume that $x_1 \neq x_2$, $y_1 \neq y_2$, and $z_1 \neq z_2$, which means that we can apply the claim.

If not both of $z_1$ and  $z_2$ occur in other questions,
we are back in Case~1, where we replace type $A$ with type $C$.
Thus, let $q_4$ be another question that contains $z_1$, and 
let $q_5$ be another question that contains $z_2$. 
Both these questions must be 
different from $q_1$, as otherwise $q_4=(x_1,y_2,z_1)$ and $q_5=(x_1,y_2,z_2)$, which by Claim~(b) would distinguish $s_1$ and $s_2$. 

For $q_4$ and $q_5$ to not distinguish 
$s_1$ and $s_2$ it must be the case that $q_4$ contains $x_2$ and $q_5$ contains $y_1$, 
both colors also occurring in other questions, as otherwise again 
we would be in Case 1. However, then $q_2$ is of type $B$, $q_3$ of type $A$, $q_4$ of type $B$ 
and $q_5$ of type $A$, 
with $q_2q_5$ and $q_4q_3$ both forming sequences of type $BA$, which is 
forbidden by pattern~7. 
Therefore, if we are not in Case~1, we have a question distinguishing $s_1$ and $s_2$.

Note that with this case, we have finished the proof in all cases in which 
one of the colors of the secrets is not part of the double pegs of double-neighboring questions. 
We can then assume in the following that every color of the secrets only 
occurs as part of double-neighboring questions.

\case[$q_1\colon (x_1,y_2,\star)$ - $C$\\
        $q_2\colon (x_1,y_2,\star)$ - $C$]{3.2}{$q_2 = q_3$}
Since $Q$ avoids pattern~9, one of $x_2$ and $y_1$ occurs in another question. Assume w.l.o.g.\ that there exists another question $q_3$ (note that as we are in the case 
$q_2 = q_3$, we can renumber and continue with $q_3$, which now is different from $q_2$) that contains $x_2$. If $q_3$ does not distinguish $s_1$ and $s_2$, it must contain either $y_1$ or $z_1$, which leads to the following two cases.

\begin{description}
\case[$q_1\colon (x_1,y_2,\star)$ - $C$\\
        $q_2\colon (x_1,y_2,\star)$ - $C$\\
        $q_3\colon (x_2, y_1, \star)$]{3.2.1}{$q_3 $ contains $y_1$}
By the assumption above, $q_3$ is of type $C$, double-neighboring to another question $q_4$ on the two first pegs. However, $q_1q_2$ and $q_3q_4$ would be two pairs 
of double-neighboring questions of type $CC$ 
which is impossible since $Q$ avoids pattern~10.

\case[$q_1\colon (x_1,y_2,\star)$ - $C$\\
        $q_2\colon (x_1,y_2,\star)$ - $C$\\
        $q_3\colon (x_2,\star , z_1)$]{3.2.2}{$q_3 $ contains $z_1$}
Similar to Case 3.2.1, $q_3$ must be double-neighboring to a question $q_4$ on the first and
third peg. 
We now have two pairs of double-neighboring questions, $q_1q_2$ of type $CC$ and $q_3q_4$ 
of type $BB$. Since $Q$ avoids pattern~11, it is impossible that both the second and third 
peg have missing colors. W.l.o.g, we can thus assume that there is a 
question $q_5$  that contains $y_1$. It must also contain $z_2$, since $x_2$ already occurs in $q_3$  and $q_4$. By assumption, it is double-neighboring to another 
question $q_6$ on $y_1$ and $z_2$.
However, if this were the case, we would have three pairs 
of double-neighboring questions, $q_1q_2$ of type $CC$, $q_3q_4$ of type $BB$ and $q_5q_6$ of 
type $AA$, which is forbidden by pattern~12. Therefore, 
this case cannot occur either.
\end{description}
\end{description}
\end{description}

In conclusion, if a strategy contains only questions of types $A$, $B$, and $C$ and avoids all of the 12~patterns listed, it must be feasible.\qed
\end{proof}

\begin{remark}
Let a strategy have only questions of type $A$, $B$ or $C$. Then
by Lemma~\ref{lem:super_lemma},
avoiding the patterns 1 to 12 is sufficient for the feasibility of a strategy.
On the other hand, it is easy to see that
if a strategy contains one of the patterns (except possibly pattern 2), it is not feasible.
\end{remark}

We end this section by noting a lower bound for $(a,b,c)$-Mastermind when $b$ and $c$ 
are not too far apart, using results from C\'{a}ceres \emph{et al.}~\cite{CHMPPSW07}.

\begin{lemma}
    \label{lem:lower_bound_2b_greater_c}
    Let $a,b,c \in \NN $ with $a \le b \le c$ and $2b > c$. Then it holds that $f(a,b,c) \ge \lfloor \frac{2}{3}
    \left(b+c-1\right) \rfloor$.
\end{lemma}

\begin{proof}
    First, by \cite[Theorem 6.1]{CHMPPSW07}, the metric dimension of $K_b \times K_c$ is $\lfloor \frac{2}{3} \left(b+c-1\right) \rfloor$, therefore $f(1,b,c) = \lfloor \frac{2}{3}  \left(b+c-1\right) \rfloor$. 
    Moreover, by \cite[Corollary 3.2]{CHMPPSW07}, the metric dimension of $K_a \times K_b \times K_c$ is greater than or equal to $K_b \times K_c$ and therefore $f(1,b,c) \le f(a,b,c)$ holds. The assertion follows.
\end{proof}

\section{General lower bound}
\label{sec_lower}
In this section, we prove a general lower bound, which is formulated in the following theorem.

\begin{theorem}
\label{thm_nearopt}
    Let $a,b,c \in \NN$. Then $f(a,b,c) \ge \lfloor \frac{a+b+c}{2}\rfloor - 1$ holds.
\end{theorem}

\begin{proof}
    Assume that $Q$ is a feasible strategy for $(a,b,c)$-Mastermind, with $x$ colors that occur only 
    once on the first peg, $y$ colors that occur only once on the second peg and $z$ colors that occur 
    only once on the third peg. Assume also that $|Q| = \lfloor \frac{a+b+c}{2}\rfloor - k$ with $k \ge 0$, and thus
    \begin{eqnarray}
    \label{abc_eq}
    2(a+b+c) & \ge & 4|Q| +4k. 
    \end{eqnarray}
    By Lemma~\ref{lem-forbidden}\ref{lem-forbiddena}, on each peg at most one color is missing. 
    \begin{description}
        \case 1{On each peg one color is missing}
        
        Recall that $x$ is the number of colors that occur exactly once on the first peg. 
        As all other colors except for the missing one occur at least twice, $|Q| \ge x + 2(a-x-1)$, and thus $x \ge 2a - |Q| -2$.
        We obtain in the same way that $y \ge 2b - |Q| - 2$ and $z \ge 2c - |Q| - 2$.
        
        Thus, by~\eqref{abc_eq},
        \begin{eqnarray*}
        x+y+z \; \ge \; 2(a+b+c) - 3|Q| - 6 \; \ge \; |Q| + 4k - 6.
        \end{eqnarray*}
        
        If $x+y+z \ge |Q|+1$, there must be a $(1,1,\star)$, a $(1,\star,1)$ or a $(\star,1,1)$-question, 
        but by Lemma~\ref{lem-forbidden}\ref{lem-forbiddend} 
        this is impossible, since there is a missing color on each peg. 
        So $x+y+z \le|Q|$ and thus $4k \le 6$, which shows that $k \le 1$.

        \case 2{There are two pegs with one color missing on each}

        W.l.o.g., the missing colors occur on the first and the second peg and the third peg has no missing color.
        Analogously to Case 1, for the first two pegs we have $x \ge 2a - |Q| -2$ and $y \ge 2b - |Q| -2$. 
        On the third peg, since there is no missing color, $|Q| \geq z + 2(c-z)$ holds and $z \geq 2c - |Q|$
        follows. 
        
        Thus, by~\eqref{abc_eq},
        \begin{eqnarray*}
        x+y+z \; \ge \; 2(a+b+c) - 3|Q| - 4 \; \ge \; |Q| + 4k - 4.
        \end{eqnarray*}
        
        Since there is a missing color on the first and on the second peg, 
        by Lemma~\ref{lem-forbidden}\ref{lem-forbiddend}, 
        no $(1,1,\star)$-question can occur, and 
        by Lemma~\ref{lem-forbidden}\ref{lem-forbiddene}, 
        it is not possible that there is both a $(1,\star,1)$-question and
        a $(\star,1,1)$-question.
        Furthermore, by Lemma~\ref{lem-forbidden}\ref{lem-forbiddenb},
        if one of these questions occurs, there is no other question of the same type.
        If $x+y+z \ge |Q|+2$, one of these situations would have to be the case, which shows that $x+y+z \le |Q| + 1$. Thus, we must have $4k - 4 \le 1$ 
        and so $k \le 1$. 

        \case 3{There is at most one peg having a missing color}
        
        W.l.o.g., there are no missing colors occurring on the second and the third peg.
        As above, this means $y \ge 2b - |Q|$ and $z \ge 2c - |Q|$. As there
        there might be a missing color on the first peg, we have $x \geq 2a - |Q| - 2$.
        
        Thus, by~\eqref{abc_eq},
        \begin{eqnarray*}
        x+y+z \; \ge \; 2(a+b+c) - 3|Q| - 2 \; \ge \; |Q| + 4k - 2.
        \end{eqnarray*}
        
        By  Lemma~\ref{lem-forbidden}\ref{lem-forbiddenb}, at most one
        $(1,1,\star)$-question, at most one  $(1,\star.1)$-question, and at most one 
        $(\star.1,1)$-question occurs. 
        By Lemma~\ref{lem-forbidden}\ref{lem-forbiddenc},
        it is not possible that all of them occur at the same time. 
        If $x+y+z \ge |Q|+3$, this would be the case, which shows that $x+y+z \le |Q| + 2 $, $4k - 2 \le 2$ and so $k \le 1$. 
        
    \end{description}

    Thus, the feasibility of $Q$ implies that $k \le 1$, which proves that $|Q| \ge \lfloor \frac{a+b+c}{2}\rfloor - 1$ for all feasible strategies.
\end{proof}

\section{The case of $\mathbf{(a,b,c)}$-Mastermind with $ \mathbf{3a \ge b+c} $}
\label{sec_ge}

In this section, we prove the following theorem.

\begin{theorem}
    \label{thm:3a_larger}
    Let $a,b,c \in \NN$ with $a \leq b \leq c$ and $3a \geq b+c$. Then it holds that $f(a,b,c) \leq \lfloor \frac{a+b+c}{2} \rfloor$.
\end{theorem}

To establish this theorem, we define suitable feasible strategies. To do so, let $c'=c$ if $a+b+c$ is even and $c'=c-1$ if $a+b+c$ is odd, so that $a+b+c'$ is an even number. 
We can now define $x$, $y$ and $z$, such that the strategies to be constructed below contain $x$~questions of type~$A$, $y$~questions of type $B$, and $z$~questions of type $C$:
\[
x:= 2a - \frac{a+b+c'}{2}, \; \;y:= 2b - \frac{a+b+c'}{2}, \; \;z:= 2c' - \frac{a+b+c'}{2}.
\]

Note that $x$, $y$ and $z$ are integers, because $a + b + c' \equiv 0 \pmod{2}$, and  that $x+y+z = \lfloor \frac{a+b+c}{2} \rfloor$. Moreover, since $3a \geq b+c$ and $b \geq a$, $x$ and $y$ are non-negative. If $c' = c$ then $z$ is also non-negative. 

In the remaining case, where $c'= c-1$, $z$ is non-negative if $3(c-1) \geq a+b$. Assume now that $3(c-1) < a+b$. It follows that $c-1 < a$, because otherwise we would have $3(c-1) \geq 3a \geq b+c \geq a+b$. Together, the inequalities $a \leq b \leq c$ and $c-1 < a$ imply $a=b=c$. So, $a+b+c - 3 = 3(c-1) < a+b$ and thus $c < 3$. This leaves two cases: $a=b=c=1$, where $f(a,b,c)=0$ and the statement of the theorem holds, or $a=b=c=2$, which does not fall into the case $c' = c-1$ since $a+b+c$ is even. Therefore, disregarding the uninteresting case $a=b=c=1$, $z$ is always non-negative, too.

\subsection{Strategy for the case $a+b+c \equiv 0,1 \pmod 4 $}

In this subsection, we provide a strategy for the case $3a \geq b+c$ and $a+b+c \equiv 0,1 \pmod 4 $. 
Observe that in this case, $x$, $y$ and $z$ are even.

As mentioned earlier, the strategy consists of $x$~questions of type~$A$, $y$~questions of type~$B$, 
and $z$~questions of type~$C$. Each of these question types will form a block (recall that a block is a connected component of the strategy graph).

\begin{strategy}
({$\lfloor \frac{a+b+c}{2}\rfloor$}-strategy for $(a,b,c)$-Mastermind 
with $3a \geq b+c$, $a+b+c \equiv 0,1 \pmod{4}$ and $ (a,b,c) \notin\{(4,4,4), (4,4,5)\}$)
\label{strat1a}
\end{strategy}

\begin{itemize}[wide=0pt]
\item[First peg:]
\mbox{}
\begin{itemize}
    \item First block: colors $1,\, 2, \, \dots, \, x$.
    \item Second block: colors $x+1,\, x+1, \, x+2, \, x+2, \,
        \dots, \, x+y/2, \, x+y/2 $.
    \item Third block: colors $x+y/2+1,\, x+y/2+1, \, x+y/2+2, \, x+y/2+2, \,
        \dots, \, a, \, a $.
\end{itemize}
\item[Second peg:]
\mbox{}
\begin{itemize}
    \item First block: colors $1,\, 1,\, 2,\, 2, \, \dots, \, x/2,\, x/2$.
    \item Second block: colors $x/2+1,\, x/2+2, \, 
        \dots, \, x/2 +y-1, \, x/2+y $.
    \item Third block: colors $x/2+y+1,\, x/2+y+2, \, x/2+y+2,\, 
        \dots, \, b, \, b, \, x/2+y+1 $.
\end{itemize}
\item[Third peg:]
\mbox{}
\begin{itemize}
    \item First block: colors $1,\, 2, \, 2, \, 
        \dots, \, x/2, \, x/2, \, 1 $.
    \item Second block: colors $x/2+1,\, x/2+2, \, x/2+2, \, \dots, \, x/2+y/2,\, x/2+y/2, \, x/2+1$.
    \item Third block: colors $x/2+y/2+1,\, x/2+y/2+2, \, \dots, \,  c'-1,\, c'$.
\end{itemize}
\end{itemize}

\begin{remark}
\label{rem1a}
Strategy~\ref{strat1a} has the following characteristics:
\begin{description}
\item[(a)]
It contains three (possibly empty) blocks, each forming a cycle of questions:
one cycle consisting of $x$ questions of type~$A$,
one consisting of $y$ questions of type~$B$, and
one consisting of $z$ questions of type~$C$.
\item[(b)] Each block can contain a double-neighboring pair of questions (if $x=2$, $y=2$ 
or $z=2$). However, because of $ (a,b,c) \not \in \{(4,4,4), (4,4,5)\}$, there 
can be at most two double-neighboring pairs of questions in total.
\item[(c)] There are no missing colors on the first two pegs, and there is
a single missing color on the third peg only in the case $a+b+c \equiv 1 \pmod{4}$.
\end{description}
\end{remark}

Two examples for different triples ($a,b,c)$ are given in Table~\ref{tab3abc}(b),(c).

\begin{table}[t]
	\setbox9=\hbox{\begin{tabular}{c || c | c | c || c }
                Peg & $1$ & $2$ & $3$ & Type
                \\
                \hline
                \hline
                $q_1$ & $1$ & $1$ & $1$ & A
                \\
                \hline
                $q_2$ & $2$ & $1$ & $1$ & A
                \\
                \hline\hline
                $q_3$ & $3$ & $2$ & $2$ & B
                \\
                \hline
                $q_4$ & $3$ & $3$ & $3$ & B
                \\
                \hline
                $q_5$ & $4$ & $4$ & $3$ & B
                \\
                \hline
                $q_6$ & $4$ & $5$ & $2$ & B
                \\
                \hline\hline
                $q_7$ & $5$ & $6$ & $4$ & C
                \\
                \hline
                $q_8$ &  $5$ & $7$ & $5$ & C
                \\
                \hline
                $q_9$ &  $6$ & $7$ & $6$ & C
                \\
                \hline
                $q_{10}$ &  $6$ & $8$ & $7$ & C
                \\
                \hline
                $q_{11}$ &  $7$ & $8$ & $8$ & C
                \\
                \hline
                $q_{12}$ &  $7$ & $6$ & $9$ & C
                \end{tabular}}
    \subcaptionbox[table]{$(a,b,c)=(4,4,c)$ for $c\in\{4,5\}$ and $6$ questions.}{
    		\raisebox{\dimexpr\ht9-\height}[\ht9][\dp9]{
                \begin{tabular}{c || c | c | c}
                Peg & $1$ & $2$ & $3$ 
                \\
                \hline
                \hline
                $q_1$ & $1$ & $1$ & $1$ 
                \\
                \hline
                $q_2$ & $1$ & $1$ & $2$ 
                \\
                \hline\hline
                $q_3$ & $1$ & $2$ & $3$ 
                \\
                \hline
                $q_4$ & $2$ & $3$ & $3$ 
                \\
                \hline\hline
                $q_5$ & $3$ & $4$ & $4$ 
                \\
                \hline
                $q_6$ & $4$ & $4$ & $4$ 
                \end{tabular}}}
    \hspace*{\fill}
	\subcaptionbox[table]{$(a,b,c)=(4,6,6)$, yielding $x=0$ and $y=z=4$, and thus $8$ questions. (Note that for $(4,6,7)$ the condition $ 3a \ge b+c$ does not hold.)}{
    		\raisebox{\dimexpr\ht9-\height}[\ht9][\dp9]{
                \begin{tabular}{c || c | c | c || c }
                Peg & $1$ & $2$ & $3$ & Type
                \\
                \hline
                \hline
                $q_1$ & $1$ & $1$ & $1$ & B
                \\
                \hline
                $q_2$ & $1$ & $2$ & $2$ & B
                \\
                \hline
                $q_3$ & $2$ & $3$ & $2$ & B
                \\
                \hline
                $q_4$ & $2$ & $4$ & $1$ & B
                \\
                \hline\hline
                $q_5$ & $3$ & $5$ & $3$ & C
                \\
                \hline
                $q_6$ & $3$ & $6$ & $4$ & C
                \\
                \hline
                $q_7$ & $4$ & $6$ & $5$ & C
                \\
                \hline
                $q_8$ &  $4$ & $5$ & $6$ & C
                \end{tabular}}}
    \hspace*{\fill}
	\subcaptionbox[table]{$(a,b,c)=(7,8,c)$ for $c\in\{9,10\}$, yielding $x=2$, $y=4$, $z=6$ and thus $12$ questions.}{
                \box9}
    \caption{Three examples of strategies for $(a,b,c)$-Mastermind with 
    $3a \ge b+c$ and $a+b+c \equiv 0,1 \pmod 4 $.\label{tab3abc}}
\end{table}

\begin{lemma}
    \label{lem:0mod4}
    Strategy~\ref{strat1a} is feasible for $(a,b,c)$-Mastermind where $a \leq b \leq c$, $3a \geq b+c$, $a+b+c \equiv 0,1 \pmod{4}$ and $(a,b,c) \not \in \{(4,4,4),(4,4,5)\}$.
\end{lemma}

\begin{proof}
We use Lemma~\ref{lem:super_lemma} to prove that 
Strategy~\ref{strat1a} is feasible.
By Remark~\ref{rem1a}(a), the condition of having only questions of type $A$, $B$ or $C$ is fulfilled, and by Remark~\ref{rem1a}(c) the patterns 1, 2, 3, 9 and 11 do not occur.

Since no questions of different types are neighboring to each other, the 
patterns 5, 6, 7 and 8 do not occur.

If two questions are double-neighboring, then the strategy contains no other questions of the same type. Therefore, the patterns 4 and 10 do not occur.

Lastly, note that since $(a,b,c) \not \in \{(4,4,4),(4,4,5)\}$ the strategy does not consist of three double-neighboring pairs of questions.
To show this, let $x=y=z=2$ hold.
Indeed, since $x = 2a - (a+b+c')/2 = 2a - (x+y+z)$, it follows that $a = (2x+y+z)/2 = 4$. 
Similarly, it follows that $b = 4$ and $c' = 4$. 
Since we excluded this case, there are at most two double-neighboring questions and pattern 12 does not occur.

Therefore, none of the forbidden patterns of 
Lemma~\ref{lem:super_lemma} occurs, and $Q$ is feasible.
\end{proof}

\begin{remark}
    As stated in Lemma~\ref{lem:0mod4}, Strategy~\ref{strat1a} cannot be applied to $(4,4,4)$ and 
    $(4,4,5)$. Note, however, that these cases can nevertheless be solved with~$6$ questions. Indeed, the 
    strategy described in Table~\ref{tab3abc}(a) is feasible for both of these cases, too.
    For the feasibility we refer to a computer program which checks the feasibility 
    of strategies by exhaustive search~\cite{SC22}.
\end{remark}

\subsection{Strategy for the case $a+b+c \equiv 2,3 \pmod 4 $}
\label{ssec_23}

We now provide a strategy for the case $3a \geq b+c$ and $a+b+c \equiv 2,3 \pmod 4 $.
Observe that in this case, $x$, $y$ and $z$ are odd. 

As before, the strategy has $x$~questions of type~$A$, $y$~questions of type~$B$, 
and $z$~questions of type~$C$. The first two blocks consist of $x-1$ and $y-1$~questions of type $A$ and $B$, respectively. The third block starts with one question of type~$A$ and one of type~$B$, and continues with $z$~questions of type $z$.

\begin{strategy}
({$\lfloor \frac{a+b+c}{2}\rfloor$}-strategy for $(a,b,c)$-Mastermind 
with $3a \geq b+c$, 
and $a+b+c \equiv 2,3 \pmod{4}$)
\label{strat1b}
\begin{itemize}[wide=0pt]
\item[First peg:]
\mbox{}
\begin{itemize}
\item First block: colors $1,\, 2, \, \dots, \, x-1$.
\item Second block: colors $x,\, x, \, x+1, \, x+1, \,
    \dots, \, x-1+(y-1)/2, \, x-1+(y-1)/2 $.
\item Third block: colors $x+(y-1)/2,\, x+(y-1)/2+1, 
\, x+(y-1)/2+1, \, x+(y-1)/2+2, \, \, x+(y-1)/2+2, \,
    \dots, \, a, \, a $.
\end{itemize}
\item[Second peg:]
\mbox{}
\begin{itemize}
\item First block: colors $1,\, 1,\, 2,\, 2, \, \dots, \, (x-1)/2,\, (x-1)/2$.
\item Second block: colors $(x-1)/2+1,\, (x-1)/2+2, \, 
    \dots, \, (x-1)/2 +y-2, \, (x-1)/2+y-1 $.
\item Third block: colors $(x-1)/2+y,\, (x-1)/2+y+1, \, 
(x-1)/2+y+2,\, (x-1)/2+y+2, \, 
    \dots, \, b, \, b, \, (x-1)/2+y $.
\end{itemize}

\pagebreak[3]

\item[Third peg:]
\mbox{}
\begin{itemize}
\item First block: colors $1,\, 2, \, 2, \, 
    \dots, \, (x-1)/2, \, (x-1)/2, \, 1 $.
\item Second block: colors $(x-1)/2+1,\, (x-1)/2+2, \, (x-1)/2+2, \, \dots, \, (x-1)/2+(y-1)/2,\, (x-1)/2+(y-1)/2, \, (x-1)/2+1$.
\item Third block: colors $(x-1)/2+(y-1)/2+1,\, (x-1)/2+(y-1)/2+1, \, (x-1)/2+(y-1)/2+2, \, (x-1)/2+(y-1)/2+3, \, \dots, \,  c'-1,\, c'$.
\end{itemize}
\end{itemize}
\end{strategy}

\begin{remark}
\label{rem2a}
Strategy~\ref{strat1a} has the following characteristics:
\begin{description}
\item[(a)]
It contains three (possibly empty) blocks, each forming a cycle of questions. The first consists of $x-1$ questions of type $A$ and
the second of $y-1$ questions of type $B$. The third block contains 
one question of type $A$, one of type $B$, and $z$ of type $C$.
\item[(b)] The first and the second block can each contain a double-neighboring pair of questions, but there cannot be any other double-neighboring pairs.
\item[(c)] There is no missing color on the first two pegs, and 
only one missing color on the third peg for the case $a+b+c \equiv 3 \pmod{4}$.
\end{description}
\end{remark}

\begin{table}
\setbox9=\hbox{\begin{tabular}{c || c | c | c || c}
            Peg & $1$ & $2$ & $3$ & Type
            \\
            \hline
            \hline
            $q_1$ & $1$ & $1$ & $1$ & A
            \\
            \hline
            $q_2$ & $2$ & $1$ & $1$ & A
            \\
            \hline\hline
            $q_3$ & $3$ & $2$ & $2$ & B
            \\
            \hline
            $q_4$ & $3$ & $3$ & $3$ & B
            \\
            \hline
            $q_5$ & $4$ & $4$ & $3$ & B
            \\
            \hline
            $q_6$ & $4$ & $5$ & $2$ & B
            \\
            \hline\hline
            $q_7$ & $5$ & $6$ & $4$ & A
            \\
            \hline
            $q_8$ & $6$ & $7$ & $4$ & B
            \\
            \hline
            $q_9$ & $6$ & $8$ & $5$ & C
            \\
            \hline
            $q_{10}$ & $7$ & $8$ & $6$ & C
            \\
            \hline
            $q_{11}$ & $7$ & $9$ & $7$ & C
            \\
            \hline
            $q_{12}$ & $8$ & $9$ & $8$ & C
            \\
            \hline
            $q_{13}$ & $8$ & $6$ & $9$ & C
            \end{tabular}}
\hspace*{\fill}
    \subcaptionbox[table]{$(a,b,c)=(3,3,c)$ for $c\in\{4,5\}$, yielding $x=y=1$ and $z=3$, and thus $5$ questions.}[.45\linewidth]{
            \centering
    		\raisebox{\dimexpr\ht9-\height}[\ht9][\dp9]{
            \begin{tabular}{c || c | c | c || c }
            Peg & $1$ & $2$ & $3$ & Type
            \\
            \hline
            \hline
            $q_1$ & $1$ & $1$ & $1$ & A
            \\
            \hline
            $q_2$ & $2$ & $2$ & $1$ & B
            \\
            \hline
            $q_3$ & $2$ & $3$ & $2$ & C
            \\
            \hline
            $q_4$ & $3$ & $3$ & $3$ & C
            \\
            \hline
            $q_5$ &  $3$ & $1$ & $4$ & C
            \end{tabular}}}
\hspace*{\fill}
    \subcaptionbox[table]{$(a,b,c)=(8,9,c)$ for $c\in\{9,10\}$, yielding $x=3$ and $y=z=5$, and thus $13$ questions.}[.45\linewidth]{
            \centering
            \begin{tabular}{c || c | c | c || c}
            Peg & $1$ & $2$ & $3$ & Type
            \\
            \hline
            \hline
            $q_1$ & $1$ & $1$ & $1$ & A
            \\
            \hline
            $q_2$ & $2$ & $1$ & $1$ & A
            \\
            \hline\hline
            $q_3$ & $3$ & $2$ & $2$ & B
            \\
            \hline
            $q_4$ & $3$ & $3$ & $3$ & B
            \\
            \hline
            $q_5$ & $4$ & $4$ & $3$ & B
            \\
            \hline
            $q_6$ & $4$ & $5$ & $2$ & B
            \\
            \hline\hline
            $q_7$ & $5$ & $6$ & $4$ & A
            \\
            \hline
            $q_8$ & $6$ & $7$ & $4$ & B
            \\
            \hline
            $q_9$ & $6$ & $8$ & $5$ & C
            \\
            \hline
            $q_{10}$ & $7$ & $8$ & $6$ & C
            \\
            \hline
            $q_{11}$ & $7$ & $9$ & $7$ & C
            \\
            \hline
            $q_{12}$ & $8$ & $9$ & $8$ & C
            \\
            \hline
            $q_{13}$ & $8$ & $6$ & $9$ & C
            \end{tabular}}
\hspace*{\fill}
\caption{Two examples of strategies for $(a,b,c)$-Mastermind with 
    $3a \ge b+c$ and $a+b+c \equiv 2,3 \pmod 4 $.\label{tab4abc}}
\end{table}

\begin{lemma}
    \label{lem:2mod4}
    Strategy~\ref{strat1b} is feasible for $(a,b,c)$-Mastermind where $a \leq b \leq c$, $3a \geq b+c$ and $a+b+c \equiv 2,3 \pmod{4}$.
\end{lemma}

\begin{proof}
We will use Lemma~\ref{lem:super_lemma} to prove that Strategy~\ref{strat1b} is feasible.
By Remark~\ref{rem2a}(a), the condition of having only questions of 
types $A$, $B$ or $C$ is fulfilled.
By Remark~\ref{rem2a}(c), none of the patterns 1, 2, 3, 9 and 11 can occur, and Remark~\ref{rem2a}(b) excludes pattern 12. Moreover, since there are no double 
neighboring questions of type $C$, and the only possible missing color is on the third peg, which is the single peg of type $C$, pattern 8 does not occur.

Since all the questions of type~$A$ except one form a cycle in the first block, and all those of type~$B$ except one form a cycle in the second block, the patterns~4 and~10 do not occur.

Lastly, the only question of type~$A$ that is not neighboring to another one of type~$A$ is 
in the cycle of the third block, which contains only one question of type~$A$. Similarly, 
the only question of type~$B$ that is not neighboring to another one is in the cycle of 
the third block, which contains only one question of type~$B$. 
Therefore, none of the patterns 5, 6 and 7 occurs.

In summary, we have excluded all of the forbidden patterns of Lemma~\ref{lem:super_lemma}, which proves that the
strategy is feasible.
\end{proof}

This lemma allows us to conclude the proof of Theorem~\ref{thm:3a_larger}.

\begin{proof}[Proof of Theorem~\ref{thm:3a_larger}.]
    Recall that $a \leq b \leq c$ and $3a \le b+c$.

    Our computer program~\cite{SC22} shows that the strategy of Table~\ref{tab3abc}(a), which has $\lfloor \frac{a+b+c}{2} \rfloor=6$ questions, is feasible for the cases $(4,4,4)$ and $(4,4,5)$.
    Lemmas~\ref{lem:0mod4} and~\ref{lem:2mod4} prove for all other cases 
    that Strategy~\ref{strat1a} is feasible for $ a+b+c \equiv 0,1 \pmod{4} $ 
    and Strategy~\ref{strat1b} is feasible for $ a+b+c \equiv 2,3 \pmod{4}$, 
    both of which consist of $\lfloor \frac{a+b+c}{2} \rfloor$ questions. This covers all cases, showing that $f(a,b,c) \le \lfloor \frac{a+b+c}{2} \rfloor$.
\end{proof}

Note that there is a gap of $1$ between the lower bound of Theorem~\ref{thm_nearopt} 
and the upper bound given by Theorem~\ref{thm:3a_larger}. Indeed, 
in the following sections, we will show that it is possible to obtain better 
strategies in some cases. 

\section{The cases of $\mathbf{(a,a,2a-1)}$-Mastermind and $\mathbf{(a,a,2a)}$-Mastermind} 
\label{sec_aa2a}

In this section, we will give a strategy with $2a-1$ questions for 
both $(a,a,2a-1)$-Mastermind and for $(a,a,2a)$-Mastermind, 
and we will show that this strategy is optimal in both cases. 
Note that in the case $(a,a,2a-1)$ the strategy of Section~\ref{ssec_23} already has $2a-1$ questions, and the strategy presented in this section gives an 
alternative way to reach a feasible strategy with this number of questions.

\begin{strategy}
($ (2a-1)$-strategy for $(a,a,2a-i)$-Mastermind, where $i\in\{0,1\}$)
\label{strat2}
\begin{enumerate}[topsep=0pt]
\item First peg: colors
$1, \, 1,\, 2, \, 2, \, \dots, \, a-1, \, a-1, a $.
\item Second peg: colors $1,\, 2, \, 2, \, \dots, \, a-1, \, a-1, \, a, \, a $.
\item Third peg: colors
$1,\, 2, \, 3, \, \dots, \, 2a-2, \, 2a-1 $.
\end{enumerate}
\end{strategy}

\begin{remark}
Strategy~\ref{strat2} has the following characteristics:
\begin{description}
\item[(a)] It consists of one path consisting of a question $q_1$ of type $D$, $2a-3$ questions $q_2, q_3,\dots,q_{2a-2} $ of type~$C$, and a final question $q_{2a-1}$ of type~$E$.
\item[(b)] No double-neighboring questions occur.
\item[(c)] For the case $c=2a-1$, there is no missing color. 
\item[] For the case $c=2a$, there is only one missing color, namely color $2a$ on the third peg. 
\end{description}
\end{remark}

One example is given in Table~\ref{tabaa2a}(a) 
(note that the example of Table~\ref{tabaa2a}(b) is based on Table~\ref{tabaa2a}(a)
and will be referred to in Section~\ref{ssec_2b<=c} about the case $ 3a < b + c $ and 
$ 2b \le  c$).

\begin{table}
\setbox9=\hbox{\begin{tabular}{c || c | c | c }
                Peg & $1$ & $2$ & $3$ 
                \\
                \hline
                \hline
                $q_1$ & $1$ & $1$ & $1$
                \\
                \hline
                $q_2$ & $1$ & $2$ & $2$
                \\
                \hline
                $q_3$ & $2$ & $2$ & $3$
                \\
                \hline
                $q_4$ & $2$ & $3$ & $4$
                \\
                \hline
                $q_5$ & $3$ & $3$ & $5$
                \\
                \hline
                $q_6$ & $3$ & $4$ & $6$
                \\
                \hline
                $q_7$ & $3$ & $4$ & $7$
                \\
                \hline
                $q_8$ &  $3$ & $5$ & $8$
                \\
                \hline
                $q_9$ & $3$ & $5$ & $9$
                \\
                \hline
                $q_{10}$ & $3$ & $5$ & $10$
                \end{tabular}}
    \subcaptionbox[table]{$(a,b,c)=(5,5,c)$ for $c\in\{9,10\}$, yielding $9$ questions.\label{subtabaa2a}}[.535\linewidth]{
            \centering
    		\raisebox{\dimexpr\ht9-\height}[\ht9][\dp9]{
                \begin{tabular}{c || c | c | c || c}
                Peg & $1$ & $2$ & $3$ & Type
                \\
                \hline
                \hline
                $q_1$ & $1$ & $1$ & $1$ & $D$
                \\
                \hline
                $q_2$ & $1$ & $2$ & $2$ & $C$
                \\
                \hline
                $q_3$ & $2$ & $2$ & $3$ & $C$
                \\
                \hline
                $q_4$ & $2$ & $3$ & $4$ & $C$
                \\
                \hline
                $q_5$ & $3$ & $3$ & $5$ & $C$
                \\
                \hline
                $q_6$ & $3$ & $4$ & $6$ & $C$
                \\
                \hline
                $q_7$ & $4$ & $4$ & $7$ & $C$
                \\
                \hline
                $q_8$ &  $4$ & $5$ & $8$ & $C$
                \\
                \hline
                $q_9$ & $5$ & $5$ & $9$ & $E$
                \end{tabular}}}
\hspace*{\fill}
    \subcaptionbox[table]{$(a,b,c)=(3,5,11)$, yielding $10$ questions.}[.43\linewidth]{
            \centering
    		\raisebox{\dimexpr\ht9-\height}[\ht9][\dp9]{            
                \begin{tabular}{c || c | c | c }
                Peg & $1$ & $2$ & $3$ 
                \\
                \hline
                \hline
                $q_1$ & $1$ & $1$ & $1$
                \\
                \hline
                $q_2$ & $1$ & $2$ & $2$
                \\
                \hline
                $q_3$ & $2$ & $2$ & $3$
                \\
                \hline
                $q_4$ & $2$ & $3$ & $4$
                \\
                \hline
                $q_5$ & $3$ & $3$ & $5$
                \\
                \hline
                $q_6$ & $3$ & $4$ & $6$
                \\
                \hline
                $q_7$ & $3$ & $4$ & $7$
                \\
                \hline
                $q_8$ &  $3$ & $5$ & $8$
                \\
                \hline
                $q_9$ & $3$ & $5$ & $9$
                \\
                \hline
                $q_{10}$ & $3$ & $5$ & $10$
                \end{tabular}}}
    \caption{An example of a strategy for $(a,a,2a-1)$-Mastermind and $(a,a,2a)$-Mastermind, and one example for the case $3a < b+c$ and $2b \leq c$.\label{tabaa2a}}
\end{table}

\begin{theorem}
\label{thm_aa2a}
Let $a \in \NN$. Then it holds that $f(a,a,2a)=2a-1$. 
\end{theorem}

\begin{proof}

First note that $f(a,a,2a) \ge 2a-1$ directly follows from 
Corollary~\ref{cor_c-1}. It only remains to show that $f(a,a,2a) \le 2a-1$ holds. To do so, we will show that Strategy~\ref{strat2} is feasible.

Let $Q$ be the set of questions of Strategy~\ref{strat2}.
We will use Lemma~\ref{lem:super_lemma} to prove by contradiction that this strategy is feasible
for $(a,a,2a)$-Mastermind. For this sake, assume that there are two secrets $s_1 = (x_1,y_1,z_1)$ and 
$s_2 = (x_2,y_2,z_2)$ that are not distinguished by $Q$.

Consider the strategy $Q' = Q \cup \{(a,1,2a+1)\}$ 
for $(a,a,2a+1)$-Mastermind. 
As it has only questions of type~$C$ 
(note that the last question is neighboring both to $q_1$ and $q_{2a-1}$),
the initial condition of Lemma~\ref{lem:super_lemma} is fulfilled.
As only the third peg has a missing color, and it has only one missing color, 
the patterns~1 and~2 do not occur.
As no double-neighboring questions occur, 
the patterns~4, 8, 9, 10, 11, and 12 do not occur.
Only questions of type~$C$ occur more than once, so the 
patterns~5, 6 and 7 do not occur.
Lastly, since only questions of type~$C$ occur more than once 
and there is no missing color on the first and second peg, 
pattern 3 does not occur.
Therefore, all the forbidden patterns of Lemma~\ref{lem:super_lemma} 
do not occur, and thus the strategy is feasible
for $(a,a,2a+1)$-Mastermind. 

In particular, $s_1$ and $s_2$ are valid secrets of $(a,a,2a+1)$-Mastermind 
and must be distinguished by $Q'$. 
Since we assumed  that $s_1$ and $s_2$ are not distinguished by $Q$, 
the only question that can distinguish them is $(a,1,2a+1)$. 
Moreover, since $2a+1$ is not a color of $(a,a,2a)$-Mastermind, 
it holds that one of the secrets gets a black peg from the 
first or second peg while the other one does not. W.l.o.g., 
we can then assume that $y_1 = 1$ and $y_2 \neq 1$. However, since we 
assumed that $s_1$ and $s_2$ are not distinguished 
by $Q$, $s_1$ and $s_2$ must receive the same number of black 
pegs from the question $q_1=(1,1,1) \in Q$. Since $s_1$ gets one black peg from it, 
we have $x_2 = 1$ and $x_1 \neq 1$, 
or $z_2 = 1$  and $z_1 \neq 1$.

\pagebreak[3]

\begin{description}
    \item[Case 1:] $x_2 = 1$ and $x_1 \neq 1$

\nopagebreak
    
        The question $q_2 = (1,2,2)$ gives at least one black peg for $s_2$, 
        because $x_2 = 1$, and, since it cannot distinguish $s_1$ 
        and $s_2$, it must also give at least one black peg for $s_1$. 
        Moreover, $x_1 \neq 1$ and $y_1 = 1$, so we must have $z_1 = 2$.
        
        Now consider the value of $y_2$. 
        
        If $y_2 = 2$, then $q_2$ gives 
        two black pegs for $s_2$ and only one for $s_1$, distinguishing the 
        two secrets. 
        
        If $y_2 = i \not \in [2]$, then this 
        color occurs in two questions of $Q$, $q_{2i-2} = (i-1,i,2i-2)$ 
        and $q_{2i-1} = (i,i,2i-1)$. Since $z_1 = 2$ and  $x_1$ cannot be 
        both $i-1$ and $i$, at least one of these questions will give 
        zero black pegs for $s_1$, distinguishing $s_1$ and $s_2$. 

\pagebreak[3]
        
    \item[Case 2:] $z_2 = 1$ and $z_1 \neq 1$

        \begin{description}
        \item[Case 2.1:] $x_1 = x_2$. 
        
        Consider the questions containing the color 
        $y_2 = i$. As before, it occurs in the questions  
        $q_{2i-2} = (i-1,i,2i-2)$ and $q_{2i-1} = (i,i,2i-1)$. 
        Since $x_1 = x_2$, the first peg gives the same number of black
        pegs for the two secrets for each question. As, furthermore, $z_1$ can occur in at most 
        one of $q_{2i-2}$ and $q_{2i-1}$, at least one of these two questions 
        gives more black pegs for $s_2$ than for $s_1$.
        
        \item[Case 2.2:] $x_1 \not= x_2$. 
        
        Consider the value of $x_1$. 
        
        If $x_1 = 1$, 
        then $q_1$ gives two black pegs for $s_1$ and only one 
        for $s_2$, distinguishing $s_1$ and $s_2$. 
        
        If  $x_1 = i \not \in \{1,a\}$, then there are two 
        questions containing $x_1$, namely $q_{2i-1} = (i,i,2i-1)$ 
        and $q_{2i} = (i,i+1,2i)$, and at least one of them 
        gives no black pegs for $s_2$, but at least one black peg for~$s_1$.         
        
        Lastly, if $x_1 = a$, then the only way 
        for $s_1$ and $s_2$ to not be distinguished by 
        $q_{2a-1}$ is that $y_2 = a$ and $y_1 \neq a$. 
        For $s_1$ and $s_2$ to not be distinguished by 
        $q_{2a-2}$, which gives at least one black peg for $s_2$, we must have $z_1 = 2a-2$. 
        Thus, no matter the value of $x_2$, the questions containing it will 
        give more black pegs for $s_2$ than for~$s_1$.
        \end{description}
\end{description}

In conclusion, $s_1$ and $s_2$ are distinguished in all cases, and therefore $Q$ is a feasible strategy. This concludes the proof.
\end{proof}

\begin{theorem}
\label{thm_aa2a-1}
Let $a \in \NN $.
Then $ f(a,a,2a-1) = 2a-1 $ holds.
\end{theorem}

\begin{proof}
Regarding the proof of the upper bound $ f(a,a,2a-1) \le 2a-1 $, we have shown in the proof of Theorem~\ref{thm_aa2a}
that Strategy~\ref{strat2} is feasible for $(a,a,2a)$-Mastermind. As this strategy does not contain the color
$2a$ on the third peg, it is also feasible for $(a,a,2a-1)$-Mastermind.

Regarding the proof of the lower bound $ f(a,a,2a-1) \ge 2a-1 $, 
for the sake of contradiction, assume that
$ f(a,a,2a-1) \le 2a-2 $ for some $a \in \NN $. 
By Lemma~\ref{lem-forbidden}\ref{lem-forbiddena}, on each peg,
at most one color is missing. For the third peg, this means
that indeed one color is missing and 
each color except the missing one occurs there exactly once.
We distinguish the following three cases. 
\begin{description}

\item[Case 1:] On at least one of the first two pegs, no 
color is missing. 

By symmetry, we can assume w.l.o.g.\ that no color is missing on the first peg.
By $ (a-2) \cdot 2  + 2 \cdot 1 = 2a-2$, then
at least two colors occur only once on the first peg,
meaning that there are at least two $(1,\star, 1)$ questions.
By Lemma~\ref{lem-forbidden}\ref{lem-forbiddenb} this is not possible.

\item[Case 2:] On both the first and second peg, one color is missing,
and on one of the first two pegs one color occurs only once.

W.l.o.g., on the first peg one color occurs only once,
meaning that there is at least one $(1,\star,1) $-question.
Since one color is missing on the first and the third peg,
by Lemma~\ref{lem-forbidden}\ref{lem-forbiddend} this is not possible.

\item[Case 3:] On both the first and the second peg, one color is missing,
and all other colors on the first and the second peg occur exactly twice.

All questions are $(2,2,1)$-questions. Consider a question $q$. 
Let $q'$ be its neighbor on the first peg and $q''$ be its neighbor 
in the second peg. If $q' = q''$, then $q$ and $q'=q''$ 
are double-neighboring. By Lemma~\ref{lem-forbidden}\ref{lem-forbiddenf},
this is not possible since one color is missing on each of the first two pegs.
However, by Lemma~\ref{lem-forbidden}\ref{lem-forbiddeng} $q' \neq q''$ is not possible either, since one color is missing on each of the first two pegs.
\qed
\end{description}
\end{proof}

In this section, we have shown that in the case of $(a,a,2a)$-Mastermind, the lower bound of Theorem~\ref{thm_nearopt} is reached, while in the case $(a,a,2a-1)$ (but also in the case $(a,a,a)$ given by~\cite{JD20}) the upper bound of Theorem~\ref{thm:3a_larger} is reached.

Note that for $(a,a,2a)$-Mastermind we have $3a=b+c$
and for $(a,a,2a-1)$-Mastermind (and $(a,a,a)$-Mastermind), we have $3a>b+c$. In the next section, we will generalize this observation and show that we can reach the bound of Theorem~\ref{thm_nearopt} in almost all the cases with $3a =b+c$.

\section{The case of $\mathbf{(a,b,c)}$-Mastermind with $3a = b+c$.} 
\label{sec_eq}

In this section, we prove the following theorem.

\begin{theorem}
\label{thm:3a=bc}
Let $a,b,c \in \NN$ with $a \leq b \leq c$, $3a = b+c$ and $(a,b,c) \neq (4,6,6)$. Then it holds that $f(a,b,c) = \lfloor \frac{a+b+c}{2} \rfloor -1$.
\end{theorem}

To establish this theorem, we define suitable feasible strategies. The case $(a,a,2a)$ has already been dealt with 
in Section~\ref{sec_aa2a}, so in this section we will focus on the case $b>a$ (and thus $c < 2a$). 
Note also that, in this case, it holds that $a+b+c = 4a$, i.e., $a+b+c$
is a multiple of $4$.

We define $x$, $y$ and $z$, such that the strategies to be constructed below 
contain $x$~questions of type~$A$, $y$~questions of type $B$, and $z$~questions of type~$C$, as follows:
\begin{eqnarray*}
x:= 1, \; \;y:= 2(b-a) -1, \; \;z:= 2(c-a) -1.
\end{eqnarray*}

Note that $x + y + z = 2(b + c) - 4a - 1$ and since $3a = b+c$, we have $x+y+z = 2a -1 = \lfloor \frac{a+b+c}{2}\rfloor -1$. Note also that $y$ and $z$ are odd numbers and, since we are in the case $a < b$, they are also positive. 

We now provide a strategy for the case $3a = b+c$ and $z \geq 5$. It consists of two blocks of questions, where the first block contains $ z-1 $ questions
and the second block contains $ y+2 $ questions.

\begin{strategy}
($\left( \lfloor \frac{a+b+c}{2}\rfloor - 1\right)$-strategy 
for $(a,b,c)$-Mastermind 
with $3a = b+c$ and $ z \ge 5$)
\label{strat_equal}

\begin{itemize}[wide=0pt]
\item[First peg:]
\mbox{}
\begin{itemize}
    \item First block: colors $1,\, 1, \, 2, \, 2, \dots, \, (z-1)/2, \, (z-1)/2$.
    \item Second block: colors $(z+1)/2,\, (z+1)/2,  \, (z+1)/2 +1, \, 
    (z+1)/2 + 1, \, \dots, \, a-1, \, a-1, \, a $.
\end{itemize}
\item[Second peg:]
\mbox{}
\begin{itemize}
    \item First block: colors $1,\, 2,\, 2,\,3,\,3,\,  \dots, \, (z-1)/2,\, (z-1)/2, \, 1 $.
    \item Second block: colors $(z+1)/2,\, (z+1)/2+1, \, (z+1)/2+2, \, 
        \dots, \, b-2, \, b-1, \, (z+1)/2 $.
\end{itemize}

\pagebreak[3]

\item[Third peg:]
\mbox{}
\begin{itemize}
    \item First block: colors $1,\, 2, \, 3, \, 
        \dots, \, z-2,  \, z-1 $.
    \item Second block: colors $z,\, z+1, \, z+1, \, \dots, \, c-2,\, c-2,\, c-1, \, c-1$.
\end{itemize}
\end{itemize}
\end{strategy}

\begin{remark}
\label{rem_equal}
Strategy~\ref{strat_equal} has the following characteristics:
\begin{description}
\item[(a)]
It contains two blocks, each forming a cycle of questions. The first block contains $z-1$ questions of type~$C$, and
the second block contains 
one question of type~$C$, $y$ questions of type~$B$, and one question of type~$A$.
\item[(b)] Since we are in the case $z \geq 5$, the first block contains at least $4$ questions. Furthermore, the second block contains 
at least $3$ questions. Therefore, no double-neighboring questions occur.
\item[(c)] A missing color occurs only on the second and third peg, but not on the first peg.
\end{description}
\end{remark}

Two examples for different triples ($a,b,c)$ are given in Table~\ref{tab3a=bc}(a),(b).

\begin{table}[t]
\resizebox{\linewidth}{!}{
\setbox9=\hbox{\begin{tabular}{@{}c@{}}
    \subcaptionbox[table]{$(a,b,c)=(2,3,3)$, $3$~questions.}{
            \centering
                \begin{tabular}{c || c | c | c || c }
                    Peg & $1$ & $2$ & $3$ & Type
                    \\
                    \hline
                    \hline
                    $q_1$ & $1$ & $1$ & $1$ & C 
                    \\
                    \hline
                    $q_2$ & $1$ & $2$ & $2$ & B 
                    \\
                    \hline
                    $q_3$ & $2$ & $1$ & $2$ & A 
                \end{tabular}}
                \\\\
                \begin{tabular}{c || c | c | c  || c}
                    Peg & $1$ & $2$ & $3$ & Type
                    \\
                    \hline
                    \hline
                    $q_1$ & $1$ & $1$ & $1$ & C
                    \\
                    \hline
                    $q_2$ & $1$ & $2$ & $2$ & C 
                    \\
                    \hline
                    $q_3$ & $2$ & $2$ & $3$ & C 
                    \\
                    \hline
                    $q_4$ & $2$ & $3$ & $4$ & B 
                    \\
                    \hline
                    $q_5$ & $3$ & $1$ & $4$ & A 
                \end{tabular}
\end{tabular}}
\setcounter{subtable}{0}
	\subcaptionbox[table]{$(a,b,c)=(4,5,7)$, yielding $x=1$, $y=1$, $z=5$ and thus $7$ questions.}{
            \centering
    		\raisebox{\dimexpr\ht9-\height}[\ht9][\dp9]{
                \begin{tabular}{c || c | c | c || c}
                    Peg & $1$ & $2$ & $3$ & Type
                    \\
                    \hline
                    \hline
                    $q_1$ & $1$ & $1$ & $1$ & $C$
                    \\
                    \hline
                    $q_2$ & $1$ & $2$ & $2$ & $C$
                    \\
                    \hline
                    $q_3$ & $2$ & $2$ & $3$ & $C$
                    \\
                    \hline
                    $q_4$ & $2$ & $1$ & $4$ & $C$
                    \\
                    \hline
                    \hline
                    $q_5$ & $3$ & $3$ & $5$ & $C$
                    \\
                    \hline
                    $q_6$ & $3$ & $4$ & $6$ & $B$
                    \\
                    \hline
                    $q_7$ & $4$ & $3$ & $6$ & $A$
                \end{tabular}}}
\quad
	\subcaptionbox[table]{$(a,b,c)=(6,9,9)$, yielding $x=1$, $y=5$, $z=5$ and thus $11$ questions.}{
            \centering
    		\raisebox{\dimexpr\ht9-\height}[\ht9][\dp9]{
                \begin{tabular}{c || c | c | c || c }
                    Peg & $1$ & $2$ & $3$ & Type
                    \\
                    \hline
                    \hline
                    $q_1$ & $1$ & $1$ & $1$ & C
                    \\
                    \hline
                    $q_2$ & $1$ & $2$ & $2$ & C
                    \\
                    \hline
                    $q_3$ & $2$ & $2$ & $3$ & C
                    \\
                    \hline
                    $q_4$ & $2$ & $1$ & $4$ & C
                    \\
                    \hline
                    \hline
                    $q_5$ & $3$ & $3$ & $5$ & C
                    \\
                    \hline
                    $q_6$ & $3$ & $4$ & $6$ & B
                    \\
                    \hline
                    $q_7$ & $4$ & $5$ & $6$ & B
                    \\
                    \hline
                    $q_8$ &  $4$ & $6$ & $7$ & B
                    \\
                    \hline
                    $q_9$ & $5$ & $7$ & $7$ & B
                    \\
                    \hline
                    $q_{10}$ & $5$ & $8$ & $8$ & B
                    \\
                    \hline
                    $q_{11}$ & $6$ & $3$ & $8$ & A
                    \end{tabular}}}
\quad
\subcaptionbox[table]{$(a,b,c)=(3,4,5)$, $5$~questions.}{
\centering
\raisebox{\dimexpr\ht9-\height}[\ht9][\dp9]{
\begin{tabular}{@{}c@{}}
    \subcaptionbox[table]{$(a,b,c)=(2,3,3)$, $3$~questions.}{
            \centering
                \begin{tabular}{c || c | c | c || c }
                    Peg & $1$ & $2$ & $3$ & Type
                    \\
                    \hline
                    \hline
                    $q_1$ & $1$ & $1$ & $1$ & C 
                    \\
                    \hline
                    $q_2$ & $1$ & $2$ & $2$ & B 
                    \\
                    \hline
                    $q_3$ & $2$ & $1$ & $2$ & A 
                \end{tabular}}
                \\\\
                \begin{tabular}{c || c | c | c  || c}
                    Peg & $1$ & $2$ & $3$ & Type
                    \\
                    \hline
                    \hline
                    $q_1$ & $1$ & $1$ & $1$ & C
                    \\
                    \hline
                    $q_2$ & $1$ & $2$ & $2$ & C 
                    \\
                    \hline
                    $q_3$ & $2$ & $2$ & $3$ & C 
                    \\
                    \hline
                    $q_4$ & $2$ & $3$ & $4$ & B 
                    \\
                    \hline
                    $q_5$ & $3$ & $1$ & $4$ & A 
                \end{tabular}
\end{tabular}}}}
    \caption{Four examples of strategies for $(a,b,c)$-Mastermind with $3a = b+c$.}
    \label{tab3a=bc}
\end{table}

\begin{lemma}
    \label{lem:3a=bc}
    Strategy~\ref{strat_equal} is feasible for $(a,b,c)$-Mastermind where $a \le b \le c$, $3a = b+c$, $b>a$ and $z \geq 5$.
\end{lemma}

\begin{proof}
We will use Lemma~\ref{lem:super_lemma} to prove that Strategy~\ref{strat_equal} is feasible.
By Remark~\ref{rem_equal}(a), the condition of having only questions of type $A$, $B$ or $C$ is fulfilled.

By Remark~\ref{rem_equal}(b), the patterns 4, 8, 9, 10, 11 and 12 do not occur, since there are no double-neighboring questions.

By Remark~\ref{rem_equal}(c), the patterns 1 and 2 do not occur. Moreover, since there is only one question of type~$A$ and the missing colors are on the second and third pegs (which are the double pegs of type $A$), the pattern 3 does not occur either.

Lastly, since the first block contains only questions of type~$C$ and the second block contains only one question each of types $A$ and $C$, the patterns 5, 6 and 7 cannot occur either.

Therefore, none of the forbidden patterns of Lemma~\ref{lem:super_lemma} occurs, which proves that the strategy is feasible.
\end{proof}

Using Lemma~\ref{lem:3a=bc}, Theorem~\ref{thm_aa2a} and the computer search on the triples for the case 
$z \le 3$ we can now prove Theorem~\ref{thm:3a=bc}.

\begin{proof}[Proof of Theorem~\ref{thm:3a=bc}.]
    Let $a,b,c \in \NN$ with $a \leq b \leq c$ and $3a = b+c$. 
    We consider three cases.

    \pagebreak[3]
    
    \begin{itemize}
    \item $b=a$:
    
     Then we have $c=2a$ and according to Theorem~\ref{thm_aa2a}, 
     $f(a,b,c) = 2a -1 = \lfloor \frac{a+b+c}{2} \rfloor -1$. 
     
    \item $b>a$ and $c \ge a+3$:
     
    Then $z = 2(c-a) - 1 \geq 5$ holds and according to Lemma~\ref{lem:3a=bc}, Strategy~\ref{strat_equal} 
    is feasible and consists of $\lfloor \frac{a+b+c}{2} \rfloor -1$ questions. 
    
    \item $b>a$ and $c \le a+2$:
    
    Since $z = 2(c-a) - 1 \le 3$, Strategy~\ref{strat_equal}
    is not defined for this case. However, this occurs only for finitely many triples $(a,b,c)$. Indeed, since $b \leq c$, 
    we have $3a = b+c \leq 2(a+2)$ and $a \leq 4$. The only triples with $a \leq 4$ and $3a = b+c$ 
    are $(1,1,2)$, $(2,2,4)$, $(2,3,3)$, $(3,3,6)$, $(3,4,5)$, $(4,4,8)$, $(4,5,7)$ and $(4,6,6)$. 
    The cases $(1,1,2)$, $(2,2,4)$, $(3,3,6)$ and $(4,4,8)$ are covered by the case $(a,a,2a)$ and the case $(4,5,7)$ is covered by the case $c \geq a+3$.
    The only remaining cases are the triples $(2,3,3)$, $(3,4,5)$ and $(4,6,6)$. Tables~\ref{tab3a=bc}(c) and (d) give strategies for the cases $(2,3,3)$ and $(3,4,5)$. These strategies use $3 = \frac{2+3+3}{2} -1$ and $5 = \frac{3+4+5}{2} -1$ questions, respectively.
    The computer program that can be found at~\cite{SC22} verifies that these strategies are feasible. For the case $(4,6,6)$, the same program shows by exhaustive search that no feasible strategy with $7$~questions exists.\qed
    \end{itemize}
\end{proof}

\section{The case of $\mathbf{(a,b,c)}$-Mastermind with $ \mathbf{3a < b+c} $}
\label{sec_le}

We divide this case into two subcases.

\subsection{The case of $(a,b,c)$-Mastermind with $ 3a < b+c$ and $2b \le c$}
\label{ssec_2b<=c}

The following theorem shows that if $c \ge 2b$,
the trivial lower bound $c-1$ is attained.

\begin{theorem}
Let $a,b,c \in \NN $  with $a \le b \le c$, $3a < b+c$ and $2b \le c$. Then $f(a,b,c) = c-1$.
\end{theorem}

\begin{proof}
By Corollary~\ref{cor_c-1}, the lower bound follows.

To prove the upper bound, consider first Strategy~\ref{strat2} given for $(b,b,2b)$-Mastermind with $2b-1$ questions. 

We first turn this strategy into one for $(b,b,c)$-Mastermind 
using $c-1$ questions. We achieve this by applying 
Remark~\ref{rem_increase} $c-2b$ times. We iteratively copy the last question and modify the 
color of the third peg to be the new color. 
This way, we indeed obtain a strategy for $(b,b,c)$-Mastermind using $c-1$ colors.

We then modify this strategy to a strategy for $(a,b,c)$-Mastermind using the same number of questions.
We do this by applying Lemma~\ref{lem_proj} $b-a$ times. On the first peg, we iteratively project
the largest  color to its immediate predecessor. The first projection is $Q_1(b-1,b)$ and we 
apply Lemma~\ref{lem_proj} because the questions $(b-1,b-1,2b-3)$, $(b-1,b,2b-2)$, and $(b,b,2b-1)$ 
are in the strategy and for all colors $y\in [b] $ and $z\in[c] $,
one of these questions differs from $y$ on the second peg 
and differs from $z$ on the third peg. 
Later on, when applying the projection $Q_1(x-1,x)$ for some value $x$ with $ a+1 \le x \le b-1$,  Lemma~\ref{lem_proj} will still be applicable,
since the questions  $(x,b-1,2b-3)$, $(x,b,2b-2)$ and $(x,b,2b-1)$ will be in the strategy.

Thus, we have obtained a feasible strategy for $(a,b,c)$-Mastermind using $c-1$ questions.
\end{proof}

Below follows an explicit definition of the strategy constructed in the preceding proof.

\begin{strategy}
({$(c-1)$}-strategy for $(a,b,c)$-Mastermind with $3a < b+c$, $ 2b \le c$)
\label{strat3}
\begin{enumerate}[topsep=0pt]
\item First peg: colors
$1, \, 1,\, 2, \, 2, \, \dots, \, a-1, \, a-1, \, a, \, a, \, \dots, a $. 
\item Second peg: colors
$1, \, 2, \, 2, \, \dots, \, b-1, \, b-1, \, b, \, b, \, \dots, b $. 
\item Third peg: colors
$1, \, 2, \, \dots, \, c-2, \, c-1 $. 
\end{enumerate}
\end{strategy}

One example is given in Table~\ref{tabaa2a}(b). 

\subsection{The case of $(a,b,c)$-Mastermind with $ 3a < b+c $ and $ 2b > c$}

To obtain a strategy for the case where $2b>c$, let us first consider the case in which $b + c \equiv 2 \pmod{3}$ and $a = (b+c-2)/3$. Note that for fixed 
values of $b$ and $c$ such that $b+c \equiv 2 \pmod{3}$, $(b+c-2)/3$ is 
the largest integer value for $a$ that still fits in the case  $3a < b + c$.

For the sake of the strategy, we define 
\begin{eqnarray*}
x:= \frac{2}{3} (2b-c-1), \; \; y:= \frac{2}{3} (2c-b-1).
\end{eqnarray*}
Note that $3b - (b+c+1) = 2b-c-1$ and $3c - (b+c+1) = 2c-b-1$ are multiples of $3$, 
since $b+c \equiv 2 \pmod 3$. Therefore, $x$ and $y$ are even integers. 
Moreover, $b -1 = x +y/2$ , $c -1 = y + x/2$ and 
$x+y = \lfloor \frac{2}{3} (b+c-1) \rfloor$ hold.

Now we present our strategy for this case.


\begin{strategy}
({$\lfloor \frac{2}{3} (b+c-1) \rfloor$}-strategy for $(a,b,c)$-Mastermind 
with $ b+c \equiv 2 \pmod{3}$ and $a = \frac{b+c-2}{3}$)
\label{stratbc2}
\begin{enumerate}
\item First peg: 
colors $1, \, 1,\, 2, \, 2, \, \dots, \, a-1, \, a-1, \, a, \, a $ 

(except for $(a,b,c) =(2,4,4)$: colors  $1, \, 1, \, 1, \, 2$).

\pagebreak[3]

\item Second peg: 
\begin{itemize}
    \item First block: $1,\, 2, \, \dots, \, x$.
    \item Second block: $x+1,\, x+2, \, x+2, \, 
        \dots, \, b-1, \, b-1, \, x+1 $.
\end{itemize}
\item Third peg: 
\begin{itemize}
    \item First block: colors $1,\, 2, \, 2, \, 
        \dots, \, x/2, \, x/2, \, 1 $. 
    \item Second block: colors $x/2+1,\, x/2+2, \, \dots, \, c-1$.
\end{itemize}
\end{enumerate}
\end{strategy}

\pagebreak[3]

\begin{remark}
\label{rembc2}
Strategy~\ref{stratbc2} has the following characteristics:
\begin{description}

\item[(a)] For $(a,b,c) \not= (2,4,4) $, it 
contains two (possibly empty) blocks, each forming a cycle of questions:
one cycle consisting of $x$ questions of type~$B$,
and one consisting of $y$ questions of type~$C$.

Note that it is possible that $x = 0$,  in which case the first block 
is empty, or that  $x = 2$, in which case the first block consists 
of a pair of double-neighboring questions. 
\item[(b)] Only one pair of double-neighboring questions may occur in the first block (as shown later).
\item[(c)] There are only two missing colors, one on the second peg and one on the third peg.
\end{description}
\end{remark}

Three examples are given in Table~\ref{tabbcmod2}(a),(b),(c), 
where the third one is the exceptional case $ (a,b,c)=(2,4,4)$.

\begin{lemma}
    Strategy~\ref{stratbc2} is feasible for 
    $2b > c$, $b+c \equiv 2 \pmod{3}$, $a = (b+c-2)/3$ and $c \geq 5$.
\end{lemma}

\begin{proof}
    We will use Lemma~\ref{lem:super_lemma} to prove that Strategy~\ref{stratbc2} 
    is feasible.
    By Remark~\ref{rembc2}(a), the condition of having only questions of type $A$, $B$ or $C$ is fulfilled.

    By Remark~\ref{rembc2}(c),
    the patterns 1 and 2 do not occur.

    All questions in the first block are of type~$B$, neighboring to 
    other questions of type~$B$, and the questions in the second block are of type~$C$, neighboring to other 
    questions of type~$C$. Thus, the patterns 5, 6, 7 and 8, where questions of different types are 
    neighboring to each other, do not occur.

    \begin{claim}
      $ y \ge 4$.
    \end{claim}

    \begin{proof}[Proof (Claim).]
    We have $ 2b > c$, $ b \le c$ and $ b+c \equiv 0 \pmod{2} $. 

    \begin{itemize}
    \item For $c = 5$, $c=6$, only $b=3$ and $b=5$, respectively, are possible.
    Then $  y = \frac{2}{3} (2c-b-1) = 4 $ follows.

    \item For $c \ge 7$ we have
    \begin{eqnarray*}
    y \; = \;  \frac{2}{3} (2c-b-1) \; \ge \; 
    \frac{2}{3} (2c-c-1) \; = \; 
    \frac{2}{3} (c-1) \; \ge \; 
    \frac{2}{3} \cdot 6 \; = \; 4. 
    \end{eqnarray*}
    \end{itemize}
    \let\qed\relax
    \end{proof}
    
    By the claim, the second block does not consist of a pair of 
    double-neighboring questions, and if $x = 2$ holds, 
    the first block forms the only double-neighboring question. So the 
    patterns 10, 11 and 12, containing at least two pairs of 
    double-neighboring questions, do not occur. Moreover, if $x = 2$ holds, there are only two questions of type~$B$, namely the ones in the double-neighboring pair, so pattern 4 does not occur.
    
    Since there is no missing color on the first peg, the common double peg of types $B$ and $C$, pattern 3 and 9 do not occur. 

    Therefore, none of the forbidden patterns of Lemma~\ref{lem:super_lemma} occurs, and the strategy is feasible.
\end{proof}

\begin{remark}
\label{rem:extra_cases_bc2mod3}
If $c \leq 4$, the only possible cases are $(a,b,c) = (1,2,3)$,
and $(a,b,c) = (2,4,4)$.
The corresponding strategies can be found in Table~\ref{tabbcmod2}(b),(c).
Our computer program shows that both strategies are feasible~\cite{SC22}.

\end{remark}

\begin{theorem}
    Let $a,b,c \in \NN $ with $a\leq b \leq c$ and $3a < b+c$ and $2b > c$. Then it holds that $f(a,b,c) = \lfloor \frac{2}{3} (b+c-1) \rfloor$.
\end{theorem}

\begin{proof}
First, using Lemma~\ref{lem:lower_bound_2b_greater_c}, we know that $f(a,b,c) \geq \lfloor \frac{2}{3} (b+c-1) \rfloor$.
It remains now to show that there always exists a feasible strategy matching this lower bound. 
We will use Strategy~\ref{stratbc2}.
To do so, we modify $a$, $b$ and $c$ so that the conditions of this 
strategy are fulfilled. 

\begin{itemize}
\item $b+c \equiv 2 \pmod{3}$. 

Let $b' = b$ and $c' = c$. 

\item $b+c \equiv 0 \pmod{3}$. 

Let $b' = b$ and $c' = c-1$. 

(If $b=c$, interchange the roles of the second and the third peg
so that still $ b' \le c'$ holds.)

\item $b+c \equiv 1 \pmod{3}$. 

Let $b' = b+1$ and $c' = c$. 

(If $b=c$, interchange the roles of the second and the third peg
so that still $ b' \le c'$ holds.)

\end{itemize}

Lastly, let $a' = \frac{b'+c'-2}{3}$. 
We now have $b'+ c' \equiv 2\pmod{3}$ and 
since $2b > c$, $b' \geq b$ and $c' \leq c$, 
we also have $2b' > c'$. 
We also have $3a' = b'+c'-2 < b'+c'$. 

We can now apply Strategy~\ref{stratbc2} 
to $(a',b',c')$-Mastermind and we obtain 
a strategy using $\lfloor \frac{2}{3}(b'+c' - 1) \rfloor$ questions.
We will then modify this strategy to obtain a strategy 
for $(a',b,c)$-Mastermind with $\lfloor \frac{2}{3}(b+c - 1) \rfloor$ questions.
\begin{itemize}
\item $b+c \equiv 2 \pmod{3}$. 

$b' = b$ and $c' = c$ hold. Therefore, we already have a strategy for 
$(a',b',c')$-Mastermind. 

\item $b+c \equiv 1 \pmod{3}$. 

Note that $\lfloor \frac{2}{3}(b+c - 1) \rfloor = \lfloor \frac{2}{3}(b+c) \rfloor$. 
Since $b' = b+1$ and $c'=c$, 
the strategy for $(a',b',c')$-Mastermind has the right amount of questions. Moreover, 
in Strategy~\ref{stratbc2}, there is a missing color 
on the second peg. So the strategy is also a valid strategy for 
$(a',b,c)$-Mastermind. 

\item $b+c \equiv 0 \pmod{3}$. 

Note that  $\lfloor \frac{2}{3}(b+c - 1) \rfloor = 
1 + \lfloor \frac{2}{3}(b+c-2) \rfloor$. 
So the strategy for $(a',b,c)$-Mastermind 
needs one more question than the strategy for $(a',b',c')$-Mastermind. 
Using  Remark~\ref{rem_increase} we know that we 
can obtain such a strategy by duplicating 
the last question and replacing the third peg by the color $c$. 
\end{itemize}

Therefore, we have strategies for $(a',b,c)$-Mastermind 
using $\lfloor \frac{2}{3}(b+c - 1) \rfloor$ questions.
If $a < a'$, by using Lemma~\ref{lem_proj} and by noticing that it can be 
applied to the 
strategy of Table~\ref{tabbcmod2}(c) for $(2,4,4)$-Mastermind and whenever a sequence of four questions 
of the same type occurs in a strategy, which happens in Strategy \ref{stratbc2} when $c \geq 5$, we can 
obtain strategies for $(a,b,c)$-Mastermind using $\lfloor \frac{2}{3}(b+c - 1) \rfloor$ questions. 
This matches with the lower bound and finishes the proof.
\end{proof}

In summary, we have obtained the following explicit strategy.

\begin{strategy}
({$\lfloor \frac{2}{3} (b+c-1) \rfloor$}-strategy for $(a,b,c)$-Mastermind with $3a < b+c$ and $ 2b  > c$)
\label{strat4}

\begin{enumerate}[topsep=0pt]
\item First peg: colors
$1, \, 1,\, 2, \, 2, \, \dots, \, a-1, \, a-1, \, a, \, a, \, \dots, a $, 

except in three cases:

$(a,b,c)\in\{(2,3,4),(2,4,4)\} $: 
colors $1, \, 1, \, 1, \, 2$,

$ (a,b,c) = (2,4,5)$:  colors $1, \, 1, \, 1, \, 2, \, 2$.

(Note that the triples $(1,1,3)$ and $(1,2,4)$ leading to the exceptional triple
$(1,2,3)$ do not fall in the case $3a<b+c$ and $ 2b > c $.)

\pagebreak[3]

\item For $ b+c = 2 \pmod{3}$: 
\begin{itemize}
    \item Second and third peg: use the same colors as in Strategy~\ref{stratbc2}.
\end{itemize}
\item For $ b+c = 1 \pmod{3}$: 
\begin{itemize}
    \item Set $b':=b+1$.
    \item Second and third peg: use the same colors as in Strategy~\ref{stratbc2} for $(a,b',c)$-Mastermind.
\end{itemize}

\pagebreak[3]

\item For $ b+c = 0 \pmod{3}$: 
\begin{itemize}
    \item Set $c':=c-1$.
    \item Second and third peg: use the same colors as in Strategy~\ref{stratbc2} for $(a,b,c')$-Mastermind.
    \item Add the question $ (a,x+1,c)$.
\end{itemize}
\end{enumerate}

\end{strategy}

One further example is given in Table~\ref{tabbcmod2}(d). 

\begin{table}[H]
\resizebox{\linewidth}{!}{
\setbox9=\hbox{\begin{tabular}{@{}c@{}}
    \subcaptionbox[table]{$(a,b,c)=(2,3,3)$, $3$~questions.}{
            \centering
                \begin{tabular}{c || c | c | c || c }
                    Peg & $1$ & $2$ & $3$ & Type
                    \\
                    \hline
                    \hline
                    $q_1$ & $1$ & $1$ & $1$ & C 
                    \\
                    \hline
                    $q_2$ & $1$ & $2$ & $2$ & B 
                    \\
                    \hline
                    $q_3$ & $2$ & $1$ & $2$ & A 
                \end{tabular}}
                \\\\
                \begin{tabular}{c || c | c | c  || c}
                    Peg & $1$ & $2$ & $3$ & Type
                    \\
                    \hline
                    \hline
                    $q_1$ & $1$ & $1$ & $1$ & C
                    \\
                    \hline
                    $q_2$ & $1$ & $2$ & $2$ & C 
                    \\
                    \hline
                    $q_3$ & $2$ & $2$ & $3$ & C 
                    \\
                    \hline
                    $q_4$ & $2$ & $3$ & $4$ & B 
                    \\
                    \hline
                    $q_5$ & $3$ & $1$ & $4$ & A 
                \end{tabular}
\end{tabular}}
\setcounter{subtable}{0}
	\subcaptionbox[table]{$(a,b,c)=(5,8,9)$, yielding $x=4$ and $y=6$ and thus $10$ questions.}{
            \centering
    		\raisebox{\dimexpr\ht9-\height}[\ht9][\dp9]{
 \begin{tabular}{c || c | c | c || c}
 Peg & $1$ & $2$ & $3$ & Type
 \\
 \hline
 \hline
 $q_1$ & $1$ & $1$ & $1$ & $B$
 \\
 \hline
 $q_2$ & $1$ & $2$ & $2$ & $B$
 \\
 \hline
 $q_3$ & $2$ & $3$ & $2$ & $B$
 \\
 \hline
 $q_4$ & $2$ & $4$ & $1$ & $B$
 \\
 \hline
 \hline
 $q_5$ & $3$ & $5$ & $3$ & $C$
 \\
 \hline
 $q_6$ & $3$ & $6$ & $4$ & $C$
 \\
 \hline
 $q_7$ & $4$ & $6$ & $5$ & $C$
 \\
 \hline
 $q_8$ &  $4$ & $7$ & $6$ & $C$
 \\
 \hline
 $q_9$ & $5$ & $7$ & $7$ & $C$
 \\
 \hline
 $q_{10}$ & $5$ & $5$ & $8$ & $C$
 \end{tabular}}}
\quad
\subcaptionbox[table]{$(a,b,c)=(2,4,4)$, yielding $x=2$ and $y=2$ and thus $4$~questions.}{
\centering
\raisebox{\dimexpr\ht9-\height}[\ht9][\dp9]{
\begin{tabular}{@{}c@{}}
    \subcaptionbox[table]{$(a,b,c)=(1,2,3)$, yielding $x=0$ and $y=2$ and thus $2$~questions.}{
            \centering
                \begin{tabular}{c || c | c | c }
                    Peg & $1$ & $2$ & $3$ 
                    \\
                    \hline
                    \hline
                    $q_1$ & $1$ & $1$ & $1$ 
                    \\
                    \hline
                    $q_2$ & $1$ & $1$ & $2$ 
                \end{tabular}}
                \\\\
                \begin{tabular}{c || c | c | c }
                    Peg & $1$ & $2$ & $3$ 
                    \\
                    \hline
                    \hline
                    $q_1$ & $1$ & $1$ & $1$ 
                    \\
                    \hline
                    $q_2$ & $1$ & $2$ & $1$ 
                    \\
                    \hline\hline
                    $q_3$ & $1$ & $3$ & $2$ 
                    \\
                    \hline
                    $q_4$ & $2$ & $3$ & $3$ 
                \end{tabular}
\end{tabular}}}
\quad
	\subcaptionbox[table]{$(a,b,c)=(4,8,10)$, yielding $x=4$ and $y=6$ and thus 
    $11$~questions.}{
            \centering
    		\raisebox{\dimexpr\ht9-\height}[\ht9][\dp9]{
                    \begin{tabular}{c || c | c | l }
                    Peg & $1$ & $2$ & $3$ 
                    \\
                    \hline
                    \hline
                    $q_1$ & $1$ & $1$ & $1$
                    \\
                    \hline
                    $q_2$ & $1$ & $2$ & $2$
                    \\
                    \hline
                    $q_3$ & $2$ & $3$ & $2$
                    \\
                    \hline
                    $q_4$ & $2$ & $4$ & $1$
                    \\
                    \hline
                    \hline
                    $q_5$ & $3$ & $5$ & $3$
                    \\
                    \hline
                    $q_6$ & $3$ & $6$ & $4$
                    \\
                    \hline
                    $q_7$ & $4$ & $6$ & $5$
                    \\
                    \hline
                    $q_8$ &  $4$ & $7$ & $6$
                    \\
                    \hline
                    $q_9$ & $4$ & $7$ & $7$
                    \\
                    \hline
                    $q_{10}$ & $4$ & $5$ & $8$
                    \\
                    \hline
                    \hline
                    $q_{11}$ & $4$ & $5$ & $10$ 
                    \end{tabular}}}}
 \caption{Four examples of strategies for $(a,b,c)$-Mastermind with $3a < b+c$ and $2b > c $.}
 \label{tabbcmod2}
\end{table}

\section{Discussion and Future Work}
\label{sec_future}

In this paper we have determined the metric dimension of $K_a \times K_b \times K_c$ for $a,b,c \in \NN$.
The only not completely solved case is $3a>b+c$, 
where two values  $\left\lfloor \frac{a+b+c}{2} \right\rfloor -1 $
and $\left\lfloor \frac{a+b+c}{2} \right\rfloor $ are possible, see Table~\ref{sum_tab1}.

Thus, we performed an exhaustive search by our computer program~\cite{SC22}
on all of those triples $ (a,b,c)$ with $a+b+c \le 21$, for which $f(a,b,c) $ has not been obtained by the theoretical results of this paper,
namely $3a>b+c$ and $ (a,b,c) \not \in \{(a,a,a), (a,a,2a-1)\}$.
More concretely, these are the following triples: 
\begin {eqnarray*}
(3,3,4), (3,4,4), (4,4,5), (4,4,6), (4,5,5),  (4,5,6),
\\
(5,5,6), (5,6,6), (5,
5,7), (5,5,8),  (5,6,7),  (5,6,8),
\\
(5,7,7), (6,6,7),  (6,6,8), (6,7,7), (6,6,9), (6,7,8).
\end{eqnarray*}

The computer runs showed that for all these cases, the upper bound is attained.
Thus, we make the following conjecture:

\begin{conjecture}
Let $a,b,c \in \NN$ with $a \leq b \leq c$. If $3a > b+c$ then it holds that $f(a,b,c) = \left \lfloor \frac{a+b+c}{2} \right \rfloor$.
\end{conjecture}

Naturally, one aim of our future work is to prove or disprove this conjecture and close the remaining gap of this work.






\section*{Acknowledgement}

We would like to thank Frank Drewes for many helpful discussions and proofreading of this paper.

This research was supported by the Kempe Foundation Grant No.~JCK-2022.1 (Sweden).

\bibliographystyle{elsarticle-num-names} 
\bibliography{main}

\end{document}